\documentclass[12pt]{article}

\usepackage{pifont}
\usepackage{amsmath}
\usepackage[english]{babel}
\usepackage[utf8x]{inputenc}
\usepackage[T1]{fontenc}
\usepackage{float}
\usepackage{natbib}
\usepackage{mathrsfs}
\usepackage{enumerate}
\usepackage{amsfonts,multirow}
\usepackage{comment,multirow}
\usepackage{caption}
\usepackage{amsthm}

\newtheorem{theorem}{Theorem}
\newtheorem{assumption}{Assumption}
\newtheorem{lemma}{Lemma}

\newtheorem{remark}{Remark}
\def\P{\mathbb{P}}
\def\R{\mathbb{R}}

\def \logit{\text{logit }}

\DeclareMathOperator*{\argmax}{arg\,max}
\DeclareMathOperator*{\argmin}{arg\,min}

\newcommand{\blind}{1}

\usepackage{graphicx}
\usepackage[colorinlistoftodos]{todonotes}
\usepackage[colorlinks=true, allcolors=blue]{hyperref}

\begin{document}
\def\spacingset#1{\renewcommand{\baselinestretch}%
{#1}\small\normalsize} \spacingset{1}

\if1\blind
{
  \title{\bf Constructing optimal treatment length strategies to  maximize quality-adjusted lifetimes }
%
  \author{Hao Sun\hspace{.2cm}\\
    Didi Chuxing Technology Company\\
    Ashkan Ertefaie\thanks{
    The authors gratefully acknowledge \textit{R01DA058996, R01DA048764, R61NS120240, and R33NS120240 from the National Institutes of Health}} \hspace{.2cm}\\
    Department of Biostatistics and Computational Biology, University of Rochester\\
    Luke Duttweiler \hspace{.2cm}\\
    Department of Biostatistics, Harvard T.H. Chan School of Public Health\\
    and \\
    Brent A. Johnson \\
    Department of Biostatistics and Computational Biology, University of Rochester}
  \maketitle
} \fi

\if0\blind
{
  \bigskip
  \bigskip
  \bigskip
  \begin{center}
    {\LARGE\bf Constructing optimal treatment length strategies to  maximize quality-adjusted lifetimes}
\end{center}
  \medskip
} \fi

\bigskip

\begin{abstract}
 Real-world clinical decision making is a complex process that involves
balancing the risks and benefits of treatments. Quality-adjusted lifetime is a composite outcome that combines patient quantity and quality of life, making it an attractive outcome in clinical research. We propose methods for constructing optimal treatment length strategies to maximize this outcome. Existing methods for estimating optimal treatment strategies for survival outcomes cannot be applied to a quality-adjusted lifetime due to induced informative censoring. We propose a weighted estimating equation that adjusts for both confounding and informative censoring. We also propose a nonparametric estimator of the mean counterfactual quality-adjusted lifetime survival curve under a given treatment length strategy, where the weights are estimated using an undersmoothed sieve-based estimator. We show that the estimator is asymptotically linear  and provide a data-dependent undersmoothing criterion.  We apply our method to obtain the optimal time for percutaneous endoscopic gastrostomy insertion in patients with amyotrophic lateral sclerosis.


\end{abstract}

\noindent%
\vfill


\newpage
\spacingset{1.9} 
\section{Introduction}
\addtolength{\textheight}{.5in}%
\subsection{Motivation}

In patient care, clinicians often navigate the delicate balance between risks and benefits, particularly in critical care, life-threatening situations, or managing incurable neurodegenerative disorders like Parkinson’s Disease, Alzheimer’s, and ALS. For example, in ALS, as functional decline progresses, patients and providers face the complex decision of whether to pursue a gastric feeding tube (PEG), which can extend survival but may also introduce side effects and reduce quality of life.

ALS is a rare but fatal neurodegenerative disease with no cure \citep{zarei2015comprehensive}. PEG tube insertion provides nutrition when oral intake becomes unsafe, potentially prolonging life. However, it can lead to complications such as infections, reduced mobility, respiratory issues, and discomfort at the insertion site. These risks increase significantly for patients in poor health, underscoring the importance of early intervention \citep{mathus1994percutaneous, leigh2003management}. Determining the optimal timing and target population for PEG remains an open question, which our article seeks to address using innovative methods.

\subsection{Optimal Treatment Strategies with Survival Outcomes}

Quality-adjusted lifetime is a composite outcome that integrates patients' quantity and quality of life \citep{goldhirsch1989costs, glasziou1990quality}. It represents discounted survival time, where the discount factor—a function of time and disease status—ranges from 0 to 1. Consequently, quality-adjusted lifetime is always less than or equal to survival time. Without censoring, patients with poor quality of life contribute to failure sets earlier on the adjusted time scale. However, in the presence of censoring, these patients may be under-represented in early failure sets due to induced informative censoring, which biases standard survival analysis approaches when quality-adjusted lifetime is the endpoint of interest \citep{gelber1989quality}.
\cite{zhao1997consistent} addressed this by proposing a consistent estimator for the distribution of quality-adjusted lifetime using a weighted estimating equation. The weights are reciprocal to the probability of observing an uncensored quality-adjusted lifetime and are estimated via the Kaplan-Meier method. Alternative approaches include regression-based methods \citep{fine2001joint, bang2002median, wang2006regression}, augmented inverse probability-weighted estimators \citep{van1999locally}, and maximum likelihood estimation of the marginal mark-and-time joint distribution \citep{huanglouis1998}.

Precision medicine further emphasizes individualized treatment strategies designed to optimize health outcomes \citep{ robins2004optimal, chakraborty2013statistical}. For survival outcomes, methods include regression-based approaches incorporating Cox’s model \citep{chen2001causal, tian2014simple}, doubly robust estimators \citep{hubbard2000nonparametric, robins2008estimation, bai2013doubly, zhao2014doubly, diaz2018targeted, cai2020one, rytgaard2022continuous}, Q-learning extensions for censored outcomes \citep{goldberg2012q}, and Kaplan-Meier-based survival estimators for treatment strategies \citep{jiang2017estimation}.  These methods, however, are not applicable to settings where the quality-adjusted lifetime is the endpoint of interest.  


\subsection{Our contribution}

We propose an estimating equation approach to estimate an optimal treatment length strategy  concerning quality-adjusted survival lifetime. Our method is applicable, for example, in settings where the interest is to construct individualized decision rules for when to stop or initiate therapy.  One of the challenges in such cases is the presence of many decision points. While in theory efficient estimators can be constructed, these estimators will rely on many nuisance parameters and the derivation of the efficient influence function will be cumbersome. Importantly,  the resulting estimator can be irregular with large biases when a subset of the nonparametriclly estimated nuisance functions fail to be consistent \citep{van2014targeted, benkeser2017doubly}. To address this, we propose two inverse probability weighting estimators: one using parametric models for weights and another allowing data-adaptive estimation of weights. The second approach achieves nonparametric efficiency under mild conditions without deriving the efficient influence function \citep{ertefaie2023nonparametric}.  Our  contributions are threefold.
 First, the proposed method addresses the induced censoring problem after transforming survival time into qualify adjusted lifetime and enables estimating an optimal treatment length strategy.
Second, it allows to estimate the nuisance parameters (i.e., weight functions) using an undersmoothed sieve-based  (i.e., data adaptive) technique while providing  root-n consistency, thereby achieving efficiency in a nonparametric model space.  Third, we provide both theoretical and finite sample undersmoothing criteria to  attain the best estimator for any given data, which can be of interest in its own right. 

Inverse probability weighting estimators with data-adaptive weights often face challenges in achieving asymptotic linearity and root-$n$ convergence. We overcome these issues by  using an undersmoothed highly adaptive lasso to estimate the weight functions, yielding an asymptotically linear estimator for the survival curve of quality-adjusted lifetime under a treatment regime. Our method is general and applies to various survival curve functions, including $t$-year survival probabilities and restricted quality-adjusted lifetime (RQAL). We focus on RQAL in simulation studies due to its relevance to the real data example.

\section{Preliminaries}\label{ch3-sec:method}

\subsection{Notation} \label{ch3-sec:data}

Assume that subjects are monitored at $K$ (finite) 
pre-specified landmarks indexed by $0 = l_0 < l_1 < l_2 <... < l_K \le \tau$, where $\tau$ is the end of follow-up or the upper limit of the support in the time-scale. We define the ALS quality-of-life score  at time $t$ as $V(t)$ and assume $V(t)$ does not change between two adjacent landmarks. We define $V(s) = 0$ as the absorbing state, which implies that $V(t) = 0$ for $t > s$; in our application, the absorbing state is death and the quality-of-life is assumed to be zero after death. The lifetime $T$ is defined: $T = \inf \{t: V(t) = 0\}$; it is recorded continuously and represents the time it takes for a subject to move from the initial state to the absorbing state. Define $\mathcal{Q}$ as a quality-of-life function mapping the state space $V(t)$ to the interval $[0, 1]$.  We further assume the functional form of $\mathcal{Q}$ is known and set to zero at the absorbing state, i.e.  $\mathcal{Q}(0) = 0$. With this notation, 
a subject's quality-adjusted lifetime is defined as
\begin{equation} \label{eq:qal surv}
    U = \int_{0}^{T} \mathcal{Q}\{V(t)\} dt.
\end{equation}
Let $s^{\ast}(x) = \inf \left[s: \int_{0}^{s} \mathcal{Q}\{V(t)\}dt \ge x \right]$ denote the minimum time (on the original scale) required to cross  the target quality-adjusted lifetime $x$. Let $Z_0$ denote a vector of baseline variables and $Z_j$ denote the updated time-dependent information at the $j$th landmark. Also, let $A_j \in \{0, 1\}, j = 1, ..., K$ denote treatment assignment indicator, and we  assume the treatment assignment is monotone that is once it is initialized or stopped, the individual will stay at the current status and no longer switch until the end of the study. We assume that the treatment assignment does not occur at baseline, i.e. $A_0 = 0$. Let $C$ denote a continuously recorded censoring time and $\tilde T=\min(T,C)$. Let $Y_j = I(\tilde T \geq l_j)$ be the at-risk indicator and realizes the value 1 if the subject is still at risk at landmark $l_j$ and 0 otherwise; $Y_0=1$ by definition. Observations after censoring or absorbing state are filled as $\emptyset$ values for simplicity. Our data consists of $n$ independent, identically distributed trajectories $\mathcal{O}  = [\{Y_j, Z_j, A_j\}_{j=0}^K,I(T<C),\tilde T]$.

 Let $T(x) = \min\{T, s^{\ast}(x)\}$, then define the censoring indicator at target time $x$ as $\Delta_c(x) = I\{C > T(x)\}$ and $\tilde T(x) =\min\{T(x), C\} $. We also define $l(x) = \operatorname*{argmax}\{j:  \tilde T(x) \ge l_j \}$ to denote the last landmark point that an individual was still at risk for a target quality-adjusted lifetime $x$.  We also define landmark based censoring indicators as $C_j(x) \in \{0, 1\}, j = 0,..., K$, where $C_j(x) = I\{\Delta_c(x)=0,\tilde T(x)<l_{j+1}\}$.  
 Overbars  with a subscript is used for cumulative information, e.g., the covariate history up to the $j$th landmark is denoted as $\bar{Z}_j = (Z_0, ..., Z_j)$.

Define the propensity score for treatment assignment $A_j, j = 1, ..., K$ as 
\begin{equation}\label{eq:trtmod}
    h_{a,j}(1\mid\bar{a}_{j-1},\bar{z}_j) = (1-A_{j-1})P(A_j = 1 \mid A_{j-1} = 0,  Y_j = 1, \bar{Z}_j = \bar{z}_j)+A_{j-1}.
\end{equation}
Note, $h_{a,j}(1\mid\bar{a}_{j-1},\bar{z}_j)=1$ if a subject has inserted PEG prior to $l_j$.
We define the censoring mechanism for $C_j(x),~j = 1,\ldots,K$ as 
\begin{equation}\label{eq:censor}
    h_{c,j,x}(1\mid \bar{a}_j, \bar{z}_j) = \{1-C_{j-1}(x)\} P\{C_j(x) = 1 \mid \bar{A}_{j} = \bar{a}_j, C_{j-1}(x) = 0, Y_j = 1, \bar{Z}_j = \bar{z}_j\}+C_{j-1}(x).
\end{equation}
Similar to $h_{a,j}(1\mid\bar{a}_{j-1},\bar{z}_j)$ above, we set $h_{c,j,x}(1\mid\bar{a}_j,\bar{z}_j)=1$ if a subject is censored prior to $l_j$.


\subsection{Causal Estimands and Estimating Equations}

The goal of our study to construct a dynamic treatment regime for PEG tube insertion that optimizes quality-adjusted lifetime. A  treatment length strategy for PEG tube insertion is a sequence of decision rules of whether  
to undergo surgery for PEG tube insertion given all available information up to that time point.  
We consider a class of decision rules 
\[\mathcal{G} = \left\{g: g^{\eta}(A_{j-1},Z_j)= I\left[(\eta^{\top}Z_j)^{1 - A_{j-1}} \ge 0\right], j = 1,...,K, \Vert \eta\Vert = 1\right\},\]
where $\Vert \cdot \Vert$ is the $L_2$ norm. Let $\bar{g}_K^\eta=\{g_1^\eta(A_{0},Z_1),\ldots,g_K^\eta(A_{K-1},Z_K)\}$. The decision rule $g^{\eta}_j \equiv g^{\eta}(A_{j-1},Z_j)$ triggers PEG tube insertion as soon as the linear score $\eta^{\top}Z_j$ crosses zero. The power $(1 - A_{j-1})$ in the exponent of the linear score guarantees that the regimes of interest belong to the class of monotone treatment regimes so that patients who have PEG tube inserted do not later have it removed and possibly wait to have it reinserted again at a later date \citep{strijbos2017percutaneous}.

Define $T^{\eta}$ and $U^{\eta}$ as the potential survival time  and the corresponding potential quality-adjusted lifetime, respectively, defined by \eqref{eq:qal surv} that would have been
observed if a patient's lifetime was uncensored during follow-up 
and had followed the strategy $\bar{g}_K^{\eta}$ 
for PEG tube insertion, 
i.e., $\{\bar{A}_K = \bar{g}_K^{\eta}\}$. We define the survivor distribution of quality-adjusted lifetime under the  treatment length strategy $\bar{g}_K^{\eta}$,  
$P(U^{\eta} > x) = S_U(x;\eta)$, and our target estimand as 
$f\{S_U(x;\eta)\}$, for some 
user-defined known functional of $S_U(x;\eta)$. 
When the goal is to maximize the $t$-year survival probability \citep{jiang2017estimation}, $f(x) = x$, and  when the target parameter is the restricted mean qualify adjusted lifetime, then 
$f\{S_U(x; \eta)\} = \int_{0}^{L_U} S_U(x; \eta) dt$  with a user defined upper limit  $L_U \leq \tau$. Naturally, estimating $S_U(x;\eta)$ consistently is 
a principal concern in the proposed method. 



Our procedure handles  confounding and dependent censoring as two sources of estimation bias. We leverage the inverse probability weighting strategy to adjust for these biases. Let $\bar{g}_K^{\eta}$ as $\Delta^{\eta}_{a}(x) = \prod_{j \le l(x)}I\{A_j = g_j^{\eta}\}$ denote the indicator function for treatment assignment at the target time $x$ under the  treatment length strategy $\bar{g}_K^{\eta}$. This indicates  whether the observed cumulative treatment assignment $\bar{A}_{l(x)}$ up to landmark $l(x)$ follows the 
strategy $\bar{g}_K^{\eta}$. 
Define the probability of following the regime $\bar{g}_K^{\eta}$ through 
landmark $l(x)$ as
$H_a(x, \bar{z}_{l(x)}) = \prod_{j \le l(x)} h_{a,j}(a\mid\bar{a}_{j-1},\bar{z}_j)$. 
We  propose the following estimating function for $S_U(x; \eta)$,
\begin{eqnarray} 
   \lefteqn{\mathbb{P}_n \left[ \frac{\Delta^{\eta}_{a}(x)\Delta_{c}(x)}{H_{a}(x)H_{c}(x)}\{I(U > x) - S_U(x; \eta)\} \right] =} \nonumber\\
    &&\mathbb{P}_n \left[\prod_{j\le l(x)}
    \frac{I\{A_j = g_j^{\eta}\}}{h_{a,j}(A_j\mid \bar{A}_{j-1},\bar{Z}_j)}
    \frac{I(C_j(x)=0)}{1 - h_{c,j,x}(1\mid\bar{Z}_j, \bar{A}_j)}
    \{I(U>x) - S_U(x;\eta)\} \right], \label{eq:direct}
\end{eqnarray}
where $H_a(x)=H_a(x,\bar{z}_{l(x)})$ and $H_c(x) = H_{c}(x,\bar{a}_{l(x)}, \bar{z}_{l(x)}) := \prod_{j \le l(x)}\{1 - h_{ c,j,x}(1\mid \bar{z}_j, \bar{a}_j)\}$ denote the propensity score and the uncensored probability up to landmark $l(x)$. 
We will discuss the methods to estimate $H_c$ and $H_a$  in Section~\ref{seb:est}.

\medskip

\begin{remark}[Extension to Quality-adjusted Lifetime]
    Mean restricted quality-adjusted lifetime is defined as the integrated quality-adjusted lifetime from zero to $L_U$, where $L_U$ is a user-defined upper-limit on the quality-adjusted lifetime scale.   Hence, 
\begin{equation}\label{eq:rqaldirect}
R(\eta) = \int_{0}^{L_U} S_U(x; \eta) dx.
\end{equation}
In our application, we set $L_U=91.12$ which is the 95\% quantile of the observed quality-adjusted lifetime. 
\end{remark}

\section{Estimation} \label{seb:est}
\subsection{Semiparametric and Nonparametric Estimators}

The estimating equation (\ref{eq:direct}) is infeasible due to unknown nuisance functions. We propose two estimators for the quality adjusted lifetime based on (\ref{eq:direct}). The first estimator relies on parametric modeling assumptions for nuisance parameters. Specifically, we define $\hat{S}_U^{par}(x;\eta)$ as a solution of (\ref{eq:direct}) where $h_{a,j}$ and $h_{c,j,x}$ are replaced with their corresponding parametric estimators  $\hat h^{par}_{a,j}$ and $\hat h^{par}_{c,j,x}$. The statistical properties of $\hat{S}_U^{par}(x;\eta)$ depend on the consistency of both of these estimators. 

The second estimator relaxes these modeling assumptions and allows one to model the nuisance functions using a nonparametric regression approach (i.e., highly adaptive lasso). We refer to this estimator as $\hat{S}_U^{npar}(x;\eta)$ with the estimated nuisance functions as $\hat h^{npar}_{a,j}$ and $\hat h^{npar}_{c,j,x}$.  Accordingly, an estimator of mean restricted quality-adjusted lifetime under regime $\bar{g}_K^{\eta}$ is
\begin{equation}\label{eq:rqaldirect}
\hat{R}^{\diamond} (\eta) = \int_{0}^{L_U} \hat{S}_U^{\diamond}(x; \eta) dx,
\end{equation}
where $\diamond \in \{par,npar\}$. 
Accordingly,  the optimal dynamic regime to maximize mean restricted quality-adjusted lifetime is therefore $\bar{g}_K^{\hat{\eta}_{opt}^{\diamond}}$ where $\hat{\eta}_{opt}^{\diamond}$ is the maximizer of $\hat{R}^{\diamond}(\eta)$.

\subsection{Finite Sample Bias Correction}

In Section~\ref{ch3-sec:theory}, we show that, under certain conditions, $\hat{R}^\diamond(\eta)$ is consistent for $R(\eta)$ for any given $\eta$.
 However, when the optimal treatment length strategy is estimated using data, $\hat{R}^\diamond(\hat{\eta}_{opt}^\diamond) \ge \hat{R}^\diamond(\eta_{opt}) \rightarrow R(\eta_{opt})$, which means that $\hat{R}^\diamond(\hat{\eta}_{opt}^\diamond) \rightarrow R(\eta_{opt})$ from above as $n \rightarrow \infty$, but $\hat{R}^\diamond(\hat{\eta}_{opt}^\diamond) \ge \hat{R}^\diamond(\eta_{opt})$ with finite sample size and $\hat{R}^\diamond(\hat{\eta}_{opt}^\diamond)$ would slightly overestimate $R(\eta_{opt})$. Such bias has no impact for the estimation of an optimal strategy, but may lead to insufficient coverage rate when the statistical inference of $R(\eta_{opt})$ is of interest. To mitigate this issue, we construct a smooth estimator to approximate the inverse probability weighting estimator  \citep{jiang2017estimation}.  Asymptotically, we can replace $g_j^{\eta} = I\left\{(\eta^{\top}Z_j)^{1 - A_{j-1}} > 0\right\}$ with $\Phi\left\{(\eta^{\top}Z_j)^{1 - A_{j-1}}/\nu\right\}$, where $\Phi(\cdot)$ is the standard normal distribution function and $\nu$ is a bandwidth tending to 0 as $n \rightarrow \infty$. When $\nu \rightarrow 0$, the normal distribution function obviously converges to the indicator function. The indicator function of treatment following $I(A_j = g_j^{\eta})$ can thus be replaced by $W_{A_j}^\eta = A_j \times \Phi\left\{(\eta^{\top}Z_j)^{1 - A_{j-1}}/\nu\right\} + (1 - A_j) \times \left[1 - \Phi\left\{(\eta^{\top}Z_j)^{1 - A_{j-1}}/\nu\right\}\right]$,
and thus, $\Delta^{\eta}_{a}(x)=\Pi_{j \le l(x)}I(A_j = g_j^{\eta})$, a product of indicator function, can be replaced with 
$\tilde\Delta_a^\eta(x) = \Pi_{j \le l(x)} W_{A_j}^\eta$.
Finally, with this modification, we redefine the estimating equation \eqref{eq:direct} at time $x$ as 
\begin{equation}\label{eq:smooth}
      \mathbb{P}_n \left[\frac{ \tilde\Delta_{a}^\eta(x)\Delta_{c}(x)}{H_{a}^\diamond(x)H_{c}^\diamond(x)}\{I(U > x) - S_U(x; \eta)\} \right]
\end{equation}
and denote the solution of \eqref{eq:smooth} as $\tilde{S}_U^\diamond(x;\eta)$. This is a normal approximation of indicator function which is used to correct the finite sample bias \citep{yang2017asymptotic}. We recommend and use $\nu = n^{-1/3}sd(\mathcal{Z})/K$, where $\mathcal{Z} = \{\eta^{\top}Z_j, j = 1, ..., K\}$, based on our empirical studies. Here, $sd(\mathcal{Z})$ is the sample standard deviation of the linear predictor and $K$ is the number of stages in the study. 
The extension to  restricted quality-adjusted lifetime  is straightforward by combining \eqref{eq:rqaldirect} and \eqref{eq:smooth}. Accordingly, define $\tilde{\eta}_{opt}^\diamond = \argmax_{\eta} \tilde{R}^\diamond(\eta)$ where
\begin{align} \label{eq:rqalsmooth}
    \begin{split}
        \tilde{R}^\diamond(\eta) = \int_{0}^{L_U} \tilde{S}_U^\diamond(x; \eta) dx.
    \end{split}
\end{align}

\subsection{Highly Adaptive Lasso Estimator} \label{ch3-sec:hal}

The highly adaptive lasso estimator is a data-adaptive modeling technique which is shown to be
consistent for the true regression function with a convergence rate faster than $n^{-1/3}$, up to a log factor 
\citep{benkeser2016highly, bibaut2019fast}. 
Highly adaptive lasso offers a flexible tool for estimation of nuisance parameters. 

Let $d_j$ denote the covariate dimension at time $j$ and $s \subset (1, ..., d_j)$. The first step to map any subset $s$ of the covariate vector to  indicator basis functions. For example, when the covariate is two dimensional $X = (X_1, X_2)$, $\phi_i(x) = \{\phi_{s=X_1,i}(x), \phi_{s=X_2,i}(x), \phi_{s=(X_1, X_2),i}(x)\} = \{I(x_1 \ge X_{i1}), I(x_2 \ge X_{i2}), I(x_1 \ge X_{i1}, x_2 \ge X_{i2})\}, i = 1,...,n,$ is the collection of the first- and the second-order indicator basis functions. Then, for each $j$, a regression model is fit (e.g., with logit link) given these set of basis functions denoted as $\phi_{j,s,i}$. Denote the corresponding parameters as $\beta_{j,s,i}$, then the propensity score can be represented as  
$
 \text{logit } h_{a,j,\beta} = \beta^a_{j,0}+\sum_{s \subset\{1,\ldots,d_j\}}\sum_{i=1}^{n} \beta^a_{j,s,i} \phi_{s,i},
$ 
where $ |\beta^a_{j,0}|+\sum_{s \subset\{1,\ldots,d_j\}}\sum_{i=1}^{n} |\beta^a_{j,s,i}|$ is an approximation of the sectional variation norm of $\logit h_{a,j}$. The highly adaptive lasso defines a minimum loss based estimator $\hat \beta^a_j$  as
\[
\hat \beta^a_j= \arg \min_{\beta: |\beta^a_{j,0}|+\sum_{s \in \mathcal{T}}\sum_{i=1}^{n} |\beta^a_{j,s,i}|<\lambda} \mathbb{P}_n L(\text{logit } h_{a,j,\beta}),
\]
where $L(\cdot)$ is the negative log-likelihood loss function, $\mathbb{P}_n$ is the empirical average and $\lambda$ is the sectional variation norm of $\logit h_{a,j}$ which can be obtained using cross-validation \citep{benkeser2016highly}. We denote the estimated propensity score at time $j$ as $\hat h_{a,j,\hat \beta^a}^{npar}\equiv \hat h_{a,j}^{npar}$. Similarly, the parameters associated with the censoring mechanism can be estimated, yielding $h_{c,j,x,\hat \beta^c}^{npar}\equiv \hat h_{c,j,x}^{npar}$. Additional details on the highly adaptive lasso are provided in Section \ref{app:hal} of the Supplementary Material.

\subsection{ Undersmoothing  in practice} \label{sec:undercri}

We generalize a  data adaptive undersmoothing criterion of \cite{ertefaie2023nonparametric} to our longitudinal settings. 
Define, for $j=1,2,\cdots,l(x)$,
\begin{equation}\label{eq:score}
 \tilde \lambda_{n,j}^a = \argmin_{\lambda}  B^{-1} \sum_{b=1}^B \left[ \sum_{(s,i) \in
 \mathcal{J}_n} \frac{1}{ \lVert \hat\beta^a_{j,\lambda,b} \rVert_{L_1}}
 \bigg\lvert \P_{n,b}^1 \tilde \Omega_{s,i}(\phi,\hat h_{a,j,\lambda,b} )
 \bigg\rvert \right],
\end{equation}
in which $\lVert \hat\beta^a_{j,\lambda,b} \rVert_{L_1} = \lvert \hat\beta^a_{j,\lambda,0}
\rvert + \sum_{s \subset\{1, \ldots, d_j\}} \sum_{i=1}^{n} \lvert
\hat \beta^a_{j,\lambda,s,i} \rvert$ is the $L_1$-norm of the coefficients
$\hat \beta^a_{j,\lambda,s,i}$ in the highly adaptive estimator $\hat h_{a,j,\lambda}$
for a given $\lambda$ in fold $b$, and $\tilde \Omega_{s,i}(\phi, h_{a,j,\lambda,b}^{npar}) =
\phi_{r,i}(\bar A_{j-1},\bar Z_{j}) \{A_{j} - \hat h_{a,j,\lambda,b}^{npar}(1 \mid \bar A_{j-1},\bar Z_{j})\}\{\hat h_{a,j,\lambda,b}^{npar}
(1 \mid \bar A_{j-1},\bar Z_{j}))\}^{-1}$. 

Undersmoothing can result in a reduced convergence rate for the highly adaptive lasso fit. Our asymptotic linearity proof of $\hat{R}^{npar}(\hat{\eta}^{opt})$ in Theorem \ref{ch3-theorem3} requires $\| \hat h_{a,j}^{npar} -h_{a,j}\|_2 = o_p(n^{-1/4})$ for all $j=1,2,\cdots,l(x)$, which is slower than the convergence rate of the cross-validated highly adaptive lasso fit, given by $O_p(n^{-1/3})$. This ensures that a certain degree of undersmoothing can be tolerated without compromising our desired asymptotic properties. Let $J$ represent the number of features included in the fit (i.e., $J=|\mathcal{J}_n|$). As shown by \cite{van2023higher}, the convergence rate of a highly adaptive lasso is $(J/n)^{1/2}$. To ensure the required $o_p(n^{-1/4})$, we must choose $J$ such that $J<n^{1/2}$, and let $\lambda^a_{n,J}$ be the corresponding $L_1$-norm. We propose our  criterion as:
\begin{align} \label{UIPWScore}
    \breve{\lambda}_{n,j}^a = \max(\lambda^a_{n,J}, \tilde \lambda_{n,j}^a).
\end{align}
In certain instances, the criterion may lead to excessive undersmoothing (i.e., selecting a very small $\tilde \lambda_{n,j}^a$), and the max operator aids in mitigating this issue. The same procedure can be implemented to obtain the undersmoothing parameter $\breve{\lambda}_{n,j}^c = \max(\lambda^c_{n,J}, \tilde \lambda_{n,j}^c)$ to fit the  function $\hat h_{c,j,\lambda^c}^{npar}$.

\section{Theoretical properties}\label{ch3-sec:theory}

\subsection{Coarsening assumptions }
The following assumptions are required for identification of the target parameters using the observed data.

\begin{assumption} (causal inference and censoring assumptions) \label{assump:basic}
\begin{enumerate}
    \item[]a. Consistency: $I(U > x) = I(U^{\eta} > x)$ if $\Delta^{\eta}_{a}(x)\Delta_{c}(x) = 1$;
    \item[]b. Sequential randomization assumption (SRA): given that a patient is still under observation at landmark $l_j$ and given the patient’s covariate history  $\bar{Z}_j, \bar{A}_{j-1}$ and $\bar{C}_{j-1}$, $A_j$ doesn't depend on future prognosis; also, given that a patient is still under observation at landmark $l_j$ and given the patient’s covariate history $\bar{Z}_j, \bar{A}_{j}$ and $\bar{C}_{j-1}$, $C_j$ doesn't depend on future prognosis;
    \item[]c. Strong positivity: $H_a(x, \bar{z}) > \alpha_1$ and $H_c(x, \bar{a}, \bar{z}) > \alpha_2$ for each $a, z$ and $\alpha_1, \alpha_2 > \epsilon$ with $\epsilon > 0$ for $j = 1,...,K$;
\end{enumerate}
\end{assumption}

Assumption \ref{assump:basic}a connects the observed outcome with the potential outcomes and validates the use of observed data to estimate the causal estimand. It also states that there is no interference between the observations. Assumption \ref{assump:basic}b assumes that there is no unmeasured confounding variable over time for treatment assignment and censoring process and is crucial to identify the mean outcome. Assumption \ref{assump:basic}c ensures that statistical inference of the treatment effect is possible.  These assumptions are commonly used in causal inference literature \citep{rosenbaum1983central,robins1986new}. 

\subsection{Parametric Estimation of the Nuisance Functions}

The first two theorems establish the results when the discrete hazard function for treatment assignment (\ref{eq:trtmod}) and censoring (\ref{eq:censor})  are correctly specified by finite dimensional parametric models, e.g., logistic models. 
\begin{theorem} \label{ch3-theorem1}
Under conditions (C.1)-(C.5) in Appendix~\ref{app-ch3}, if the functions $h_{a,j}$ and $h_{c,j,x},~j=1,\ldots,K,$ are correctly specified, for any regime $\bar{g}^{\eta}_K$, we have, as $n \rightarrow \infty$, (1)  $\hat{S}_U^{par}(x; \eta) \rightarrow_p S_U(x; \eta)$ for $x \in (0, L_U]$,
    (2) $n^{1/2} \{\hat{S}_U^{par}(x; \eta) - S_U(x; \eta)\}$ converges to a mean-zero Gaussian process,
    (3) $n^{1/2} \{\hat{S}_U^{par}(x; \hat{\eta}_{opt}^{par}) - S_U(x; \eta_{opt})\} \rightarrow_d N\{0, \sigma^2_{\eta_{opt}, 1}(x)\}$, where the expression of $ \sigma^2_{\eta_{opt}, 1}(x)$ can be found in Appendix~\ref{app-ch3},
    (4) $n^{1/2} \{\hat{S}_U^{par}(x; \hat{\eta}_{opt}^{par}) - \tilde{S}_U^{par}(x; \tilde{\eta}_{opt}^{par})\} = o_p(1)$.
\end{theorem}

Theorem~\ref{ch3-theorem1} serves as the building block of our main result in Theorem~\ref{ch3-theorem2}.
\begin{theorem} \label{ch3-theorem2}
Under conditions (C.1)-(C.5) in Appendix~\ref{app-ch3}, if $h_{a,j}$ and $h_{c,j,x},~j=1,\ldots,K,$, are correctly specified, for any regime $\bar{g}^{\eta}_K$, we have, as $n \rightarrow \infty$, (1) $\hat{R}^{par}(\eta) \rightarrow_p R(\eta)$,
    (2) $n^{1/2} \{\hat{R}^{par}(\eta) - R(\eta)\}$ converges to a mean-zero Gaussian process,
    (3) $n^{1/2} \{\hat{R}^{par}(\hat{\eta}_{opt}^{par}) - R(\eta_{opt})\} \rightarrow_d N(0, \sigma^2_{\eta_{opt}, 2})$, where the expression of $ \sigma^2_{\eta_{opt}, 2}$ can be found in Appendix~\ref{app-ch3},
    (4) $n^{1/2} \{\hat{R}^{par}(\hat{\eta}_{opt}^{par}) - \tilde{R}^{par}(\tilde{\eta}_{opt}^{par})\} = o_p(1)$.
\end{theorem}

\subsection{Nonparametric estimator of $S_U(x; \eta)$} \label{sec:nonparm}

 To improve the robustness of our approach, we propose a nonparametric estimator where the weight functions are estimated nonparametrically. In general, obtaining an inverse probability weighted estimator with a desired root-$n$ is challenging due to  slower rate of convergence of estimated weight functions. 
 We overcome this issue by estimating the nuisance parameters using undersmoothed highly adaptive lasso denoted as $\hat H_a^{npar}$ and $\hat H_c^{npar}$. Specifically, we  show that (i) under certain assumption listed below $\Psi_{\hat H_a^{npar},\hat H_c^{npar}}(x;\eta)$ is an unbiased estimating equation; and, (ii)  the resulting estimator $\hat{S}_U^{npar}(x; \eta)$ is asymptotically linear. 
 The estimated optimal regime at $x$ is defined as $\hat{\eta}_{opt}^{npar} = \operatorname*{argmax}_{\eta}\hat{S}_U^{npar}(x; \eta)$. 

To improve the performance of our estimator in finite samples, we propose employing cross-fitting for estimating nuisance parameters ~\citep{klaassen1987consistent, zheng2011cross, chernozhukov2017double}. The dataset is randomly divided into $B$ mutually exclusive and exhaustive sets, each approximately of size $n B^{-1}$. The empirical distribution of a training and validation sample is denoted as $\P_{n,b}^0$ and $\P_{n,b}^1$, respectively, for $b=1,2,\cdots,B$. 
The resulting estimating equation is
\begin{equation} 
   \Psi_{H_a^{npar},H_c^{npar}}(x;\eta)= \frac{1}{B} \sum_{b=1}^B \P_{n,b}^1 \frac{\Delta^{\eta}_{a}(x)\Delta_{c}(x)}{H_{a,b}^{npar}(x)H_{c, b}^{npar}(x)}\{I(U > x) - S_U(x; \eta)\}.
\end{equation}


 The next theorem shows that, when undersmoothed highly adaptive lasso estimator is used to model the nuisance parameters, we can still develop the asymptotic linearity of our estimators.

\begin{assumption}[Complexity of nuisance functions] \label{assump:cadlag}
 For each $j=1,2,\cdots,l(x)$, the functions $Q_j(x)= E\{I(U  >x) | \bar{A}_{j}=\bar g_j^\eta,C_{j-1}(x)=0,\bar{Z}_j \}$, $h_{a,j}$ and $h_{c,j,x}$ are c\`{a}dl\`{a}g with finite
  sectional variation norm.
\end{assumption}
Assumption \ref{assump:cadlag} is rather mild and includes a wide class of true conditional mean models likely encountered in practice. Specifically, it defines the overall smoothness assumption for the true functions, which is considerably less restrictive compared to local smoothness assumptions imposed by  H{\"o}lder balls~\citep[e.g.,][]{robins2008higher, robins2017minimax, mukherjee2017semiparametric}. This assumption facilitates a rapid convergence rate of $n^{-1/3}$ (up to a log factor) achievable through the highly adaptive lasso, irrespective of the dimensionality \citep{ van2023higher}.

\begin{assumption}[Undersmoothing]\label{assump:proj}
Let 
\begin{align*}
    f^a_j&= \frac{Q_j(x) I(\bar A_{j-1}=\bar g_{j-1}^\eta) I\{C_{j-1}(x)=0\} \hat h_{c,j,x}(0\mid\bar{z}_s, \bar{a}_s) }{\Pi_{s \le j}\{1 - h_{c,s}(1\mid\bar{z}_s, \bar{a}_s)\}\Pi_{s \le j}\{ h_{a,s}(g_j^\eta\mid\bar{z}_s, \bar{a}_{s-1})\}}\\
    f^c_j&= \frac{Q_j(x) I(\bar A_{j-1}=\bar g_{j-1}^\eta) I\{C_{j-1}(x)=0\} I(A_j=g_j^\eta) }{\Pi_{s \le j}\{1 - h_{c,s}(1\mid\bar{z}_s, \bar{a}_s)\}\Pi_{s \le j}\{ h_{a,s}(g_j^\eta\mid\bar{z}_s, \bar{a}_{s-1})\}},
\end{align*}
where $Q_j(x)= E\{I(U  >x) | \bar{A}_{j}=\bar g_j^\eta,C_{j-1}(x)=0,\bar{Z}_j \}$ and $j=1,2,\cdots,l(x)$.
 Let  ${f}_{\phi,j}^a$ and ${f}_{\phi,j}^c$   be the projections of  $f^a $ and $f^c$ onto a linear
  span of basis functions $\phi_{k}$ in $L^2(P)$, for $\phi_{k}$
  satisfying conditions (\ref{eq:basisa}) and (\ref{eq:basisc})  of Lemma  \ref{lem:ucondition}. Then,  $\lVert f_j^a - {f}_{\phi,j}^a
  \rVert_{2,\mu} = O_p(n^{-1/4})$ and $\lVert f_j^c - {f}_{\phi,j}^c
  \rVert_{2,\mu} = O_p(n^{-1/4})$ 
where $\mu$ is a $\sigma$-finite measure.
\end{assumption}

Assumption \ref{assump:proj} is a crucial condition for maintaining the asymptotic linearity of $ \hat{R}^{npar}(\hat{\eta}_{opt}^{npar})$. It asserts that when the estimated coarsening mechanisms (i.e., $h_{a,j}$ and $h_{c,j,x}$) are appropriately undersmoothed, the generated features can approximate functions $f_j^a$ and $f_j^c$ at a rate of $n^{-1/4}$. Under Assumption \ref{assump:cadlag}, ${f}_j^a$ and ${f}_j^c$ fall into the class of c\`{a}dl\`{a}g functions with a finite sectional variation norm, and thus can be approximated using the highly adaptive lasso-generated basis functions at a rate of $n^{-1/3}$ (up to a $\log n$ factor). It's noteworthy that the required rate in Assumption \ref{assump:proj}, $O_p(n^{-1/4})$, is slower than the $O_p(n^{-1/3})$ rate obtained by the highly adaptive lasso estimator, affirming the credibility of the assumption.

In the proof of Theorem \ref{ch3-theorem3}, we show that our estimator will be asymptotically linear when 
$\left| \P_n \sum_{j=0}^{l(x)} f^a_j \left\{I(A_j=g_j^\eta)  - \hat h_{a,j}^{npar}(g_j^\eta\mid\bar{a}_{j-1},\bar{z}_j)\right\} \right|=o_p(n^{-1/2})$ and \\
$\left| \P_n \sum_{j=0}^{l(x)} f^c_j \left\{I(C_j(x)=0)  - \hat h_{c,j,x}^{npar}(0\mid \bar{a}_j,\bar{z}_j)\right\} \right|=o_p(n^{-1/2})$.  Lemma \ref{lem:ucondition} in the Supplementary Material provides the theoretical undersmoothing prerequisites for $\hat h_{a,j}^{npar}(g_j^\eta\mid\bar{a}_{j-1},\bar{z}_j)$ and $\hat h_{c,j,x}^{npar}(0\mid \bar{a}_j,\bar{z}_j)$ to satisfy the $o_p(n^{-1/2})$ conditions.

\begin{theorem} \label{ch3-theorem3}
Let $L(\cdot)$ be a log-likelihood loss function. Assume that conditions (C.1), (C.2), (C.4) and (C.5) in Appendix~\ref{app-ch3} hold and both treatment assignment  and censoring mechanisms belong to a class of c\`{a}dl\`{a}g functions with finite sectional variation norm. Let $\hat H_{a,\lambda_n^a}^{npar}$ and $\hat H_{c,\lambda_n^c}^{npar}$ be highly adaptive lasso estimators of $H_a$ and $H_c$ using $L_1$ norm bounds equal to $\lambda_n^a$ and $\lambda_n^c$ where $\lambda_n^a$ and $\lambda_n^c$ are data dependent parameter chosen such that conditions (\ref{eq:basisa}) and (\ref{eq:basisc}) are satisfied. 
 Then  
$
n^{1/2} \{\hat{S}^{npar}(\hat{\eta}_{opt}^{npar}) - S({\eta}_{opt})\} = n^{-1/2} \sum_{i=1}^n IC_{ HAL}(o_i) + o_p(1),
$ where the influence function $IC_{ HAL}(o_i)$ is defined in Appendix~\ref{app-ch3}.
\end{theorem}
A direct consequence of Theorem \ref{ch3-theorem3} is that the RQAL with highly adaptive lasso will be asymptotic linear with
      $n^{1/2}\left\{\hat{R}^{npar}(\hat{\eta}_{opt}^{npar}) - R(\eta_{opt})\right\} = n^{-1/2} \sum_{i=1}^{n} \int_{0}^{L_U} IC_{HAL}(o_i; x, \eta_{opt}) dx + o_p(1).$
An important implication of Theorems \ref{ch3-theorem2} and \ref{ch3-theorem3} is that when the form of the influence functions is unknown, inference is attainable via the standard bootstrap \citep{cai2020nonparametric}.    Alternatively, for $\diamond \in \{par,npar\}$, one may use a conservative variance estimator for $\hat{R}^{\diamond}(\hat{\eta}_{opt}) $ defined as $\hat\sigma^2_{\hat\eta_{opt},\diamond} = n^{-1}\sum_{i=1}^{n} \left(\int_{0}^{L_U} \widehat{IC}_{x, \hat{\eta}}^{\diamond}(o_i)\right)^2$, 
where $\widehat{IC}_{x, \hat\eta}^{\diamond}(o_i) = W_{x, \hat\eta}^{\diamond}(o_i)\{I(U_i > x) - S_U(x; \hat{\eta})\} / n^{-1}\sum_{i=1}^{n} W_{x, \hat\eta}^{\diamond}(o_i)$
and \\ $W_{x, \hat\eta}^{\diamond}(o_i) = \Delta^{\hat\eta}_{a,i}(x)\Delta_{c,i}(x) \{I(U_i > x) - S_U(x; \hat{\eta})\}/\hat H_{a,\lambda_n^a i}^{\diamond}(x) \hat H_{c,\lambda_n^c, i}^{\diamond}(x)$.

\section{Simulation Results}\label{ch3-sec:simu}

We conduct  simulation studies to investigate the performance of the proposed estimator for restricted quality-adjusted lifetime (RQAL). We estimate the weight functions using logistic regression, highly adaptive lasso, and random forest where we used R packages \texttt{randomForest} \citep{liaw2002classification}  and \texttt{hal9001} \citep{coyle2019hal9001} with all tuning parameters
set to their respective default values. The only exception is the parameter $ \lambda$ in highly adaptive lasso that is specified using the undersmoothing criteria presented in Section \ref{sec:undercri}.

Let the end time of the study be $L$, the gap between two landmarks be $G$. Let  $K = L/G$ be the number of stages in the simulated longitudinal study. Define the true optimal regime as $g^{\eta_{opt}} = \{g_j^{\eta_{opt}} =I(\eta_{opt}^\top Z_j\ge 0)^{1 - A_{j-1}}, j = 1,\ldots,K\} $ where $\eta_{opt} = (1, -1, -1)$. For convenience of notation, we define $W_j = I(\bar{A}_j = g^{\eta_{opt}})$ indicating whether the observed treatment assignment $\bar{A}_j$ matches the true optimal regime up to the $j$th landmark.
We generate baseline values $Z_0 = (Z_{10}, Z_{20}, Z_{30})$ independently from uniform(0.6, 1) and set $A_0 = C_0 = 0$. 
 If a patient is still at risk at the $j$th landmark, we will update the time-dependent information $Z_j = (Z_{1j}, Z_{2j}, Z_{3j})$ where $Z_{1j} = \mathrm{uniform}(0.4 + 0.02\times K, 1)\times Z_{10}$, $Z_{2j} = \mathrm{uniform}(0.4 + 0.02\times K, 1)\times Z_{20}$ and $Z_{3j} = \mathrm{uniform}\{0.4 + 0.02 \times K - (1 - W_j) \times 0.1, 1 - (1 - W_j) \times 0.1)\}\times Z_{30}$. In our model, $Z_{3j}$ serves the role of the quality of life score assessed at time $j$. Patients who do not follow the optimal regime $g^{\eta_{opt}}$ will receive a penalty with respect to quality of life score or the hazard rate function for survival, and thus, their expected quality-adjusted lifetime will not be optimal when deviating from $\eta_{opt}$.

We assume the data generation process follows the following discrete hazard models: $h_{a, j}(\bar{z}) = logit^{-1}(\kappa_0 + \kappa_1 \times z_{1j} + \kappa_2 \times z_{2j})$ and  
        $h_{c, j}(\bar{z}, \bar{a}) = logit^{-1}(\nu_0 + \nu_1 \times z_{1j} + \nu_2 \times z_{2j})$, 
where $(\kappa_0, \kappa_1, \kappa_2) = (0.5 - 0.1 \times K, -0.5, -0.5)$,  $(\nu_0, \nu_1, \nu_2) = (-2 + 0.5 \times \kappa_0, -1, -1)$. 
We consider two scenarios  for the  treatment and censoring mechanisms:

\par
\medskip
\begin{itemize}
    \item \textit{Scenario 1} (hazard models are correctly specified by logistic models):
    In this scenario, the covariates $Z$ is observed and can be used to model the discrete hazard functions for treatment assignment and censoring.
    \item \textit{Scenario 2} (hazard models are mis-specified by logistic models):
    In this scenario, the transformed covariates $X = \{(Z_1 + Z_2 - 1)^2, (Z_1-0.5)^2, Z_3\}$ is observed and can be used model the discrete hazard functions for treatment assignment and censoring.
\end{itemize}

We generate 500 datasets of sizes  250 and 500 according to each scenario where, for $t \in (jG, (j+1)G)]$ and   $j = 0,..., K-1$, the discrete survival hazard model  $h_{y}(t, \bar{a}, \bar{z}) =  pr(T = t | T > t - 1, \bar{A}_j = \bar{a}_j, \bar{C}_j = 0, \bar{Z}_j = \bar{z})$ is  
    $h_{y}(t, \bar{a}, \bar{z}) =   \text{logit}^{-1}\{-5 + t \times (6 \times 0.75)/K -0.5 \times z_{1j} -0.5 \times z_{2j} -0.5 \times a_j \times W_j\}$.

We then the define the utility map as $\mathcal{Q}(Z_{3j}) = Z_{3j}/\max_{k = 0, .., K} Z_{3k}$ which is always a number between (0, 1]. Patients' quality of life scores will be set to zero if they are no longer at-risk. This set up will generate the data with around 15\% censoring.  To approximate the true optimal restricted quality-adjusted lifetime, we generated  $10^5$ samples from the true survival model with treatment assignment according to $g^{\eta_{opt}}$. To estimate the optimal regimes, a genetic  optimization algorithm in the R package \texttt{rgenoud} \citep{mebane2011genetic} was implemented.

When the weight functions are estimated using random forest, the proposed estimator is no longer asymptotically linear, and the theoretical formula for the standard error does not exist. We used the conservative variance estimator presented at the end of Section \ref{sec:nonparm}.

\begin{remark}
    To simplify the presentation of different estimators in the tables, we use the notation $\breve R$ and $\breve \eta$  to represent the estimated quantities corresponding to the methods specified in each row. For example, in Table \ref{tb:4.2},  $\breve R(\breve \eta_{opt})$ in row 2 denotes $\hat R^{npar}(\hat \eta^{npar}_{opt})$, whereas in row 4, it denotes $\tilde R^{par}(\tilde \eta^{par}_{opt})$.
\end{remark}


\subsection{Correctly Specified Hazard Models} \label{sim:corect}

We generate the treatment and censoring variables according to scenario 1 and fit a correctly specified logistic model for the treatment assignment and the censoring mechanism. We consider the following two setups for $L$ and $G$.   In the first set of simulation, we set $L = 60$ and $G = 10$ with the true targwt parameter $R(\eta_{opt})$ = 21.04 weeks.   In the second set of simulation, we set $L = 100$ and $G = 4$ with the true $R(\eta_{opt})$ = 31.74 weeks in this case. See Section \ref{app-scenario1} of the Supplementary Material for more details.

Table \ref{App:tb:4.1} in Section \ref{app-scenario1} of the Supplementary Material  shows the results of the estimated optimal treatment length strategy using the proposed method where weight functions are estimated using logistic regression, highly adaptive lasso and random forest.  The highly adaptive lasso and logistic regression perform similarly in terms of estimating RQAL under the estimated optimal regime (i.e., $\breve{R}(\breve{\eta}_{opt})$). Importantly, the empirical standard errors reported in the parenthesis in the column  $\breve{R}(\breve{\eta}_{opt})$ matches closely the estimated theoretical standard errors (i.e., SE) which confirm our results in Theorem \ref{ch3-theorem3}. Figure \ref{fg:4.1} (first row) displays the coverage probabilities where the dashed line represents the nominal coverage of 0.95. 
In general, both logistic (circle) and highly adaptive lasso (cross) fits lead to reasonable coverage probabilities. When $K = 25$ and $n = 250$, the estimated theoretical standard errors tend to underestimate empirical standard errors, which partially leads to the insufficient coverage probabilities. This is expected because for larger number of stages, we would need more sample to learn the optimal regime. However, when the sample size is increased to 500, the empirical standard errors match the estimated standard errors and the coverage probabilities are improved considerably. See Section \ref{app-scenario1} of the Supplementary Material for additional discussion and figures. 

\begin{figure}[t]
\centering
\captionsetup{font=small}
  \includegraphics[width=0.7\textwidth]{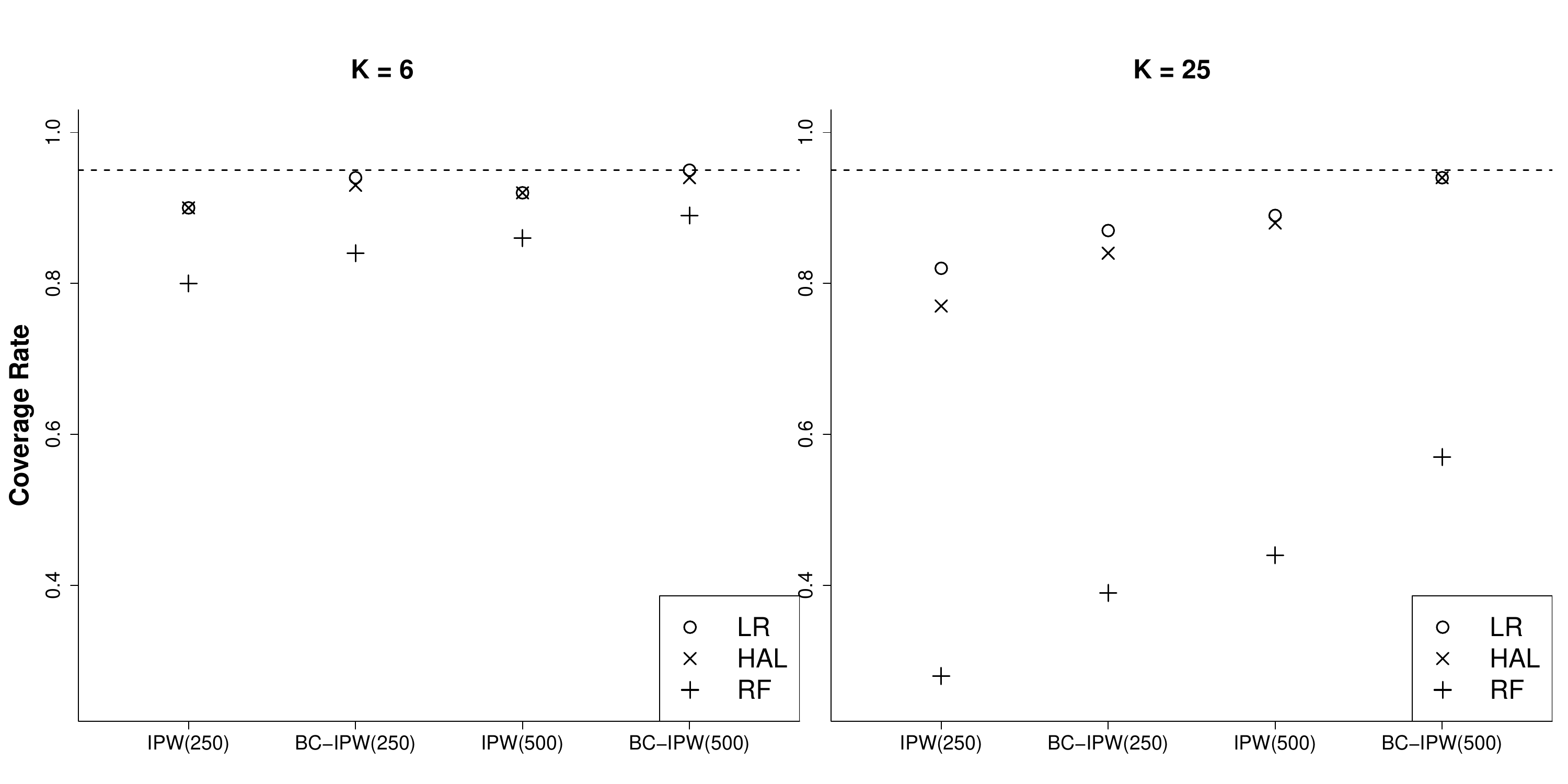}
  \includegraphics[width=0.7\textwidth]{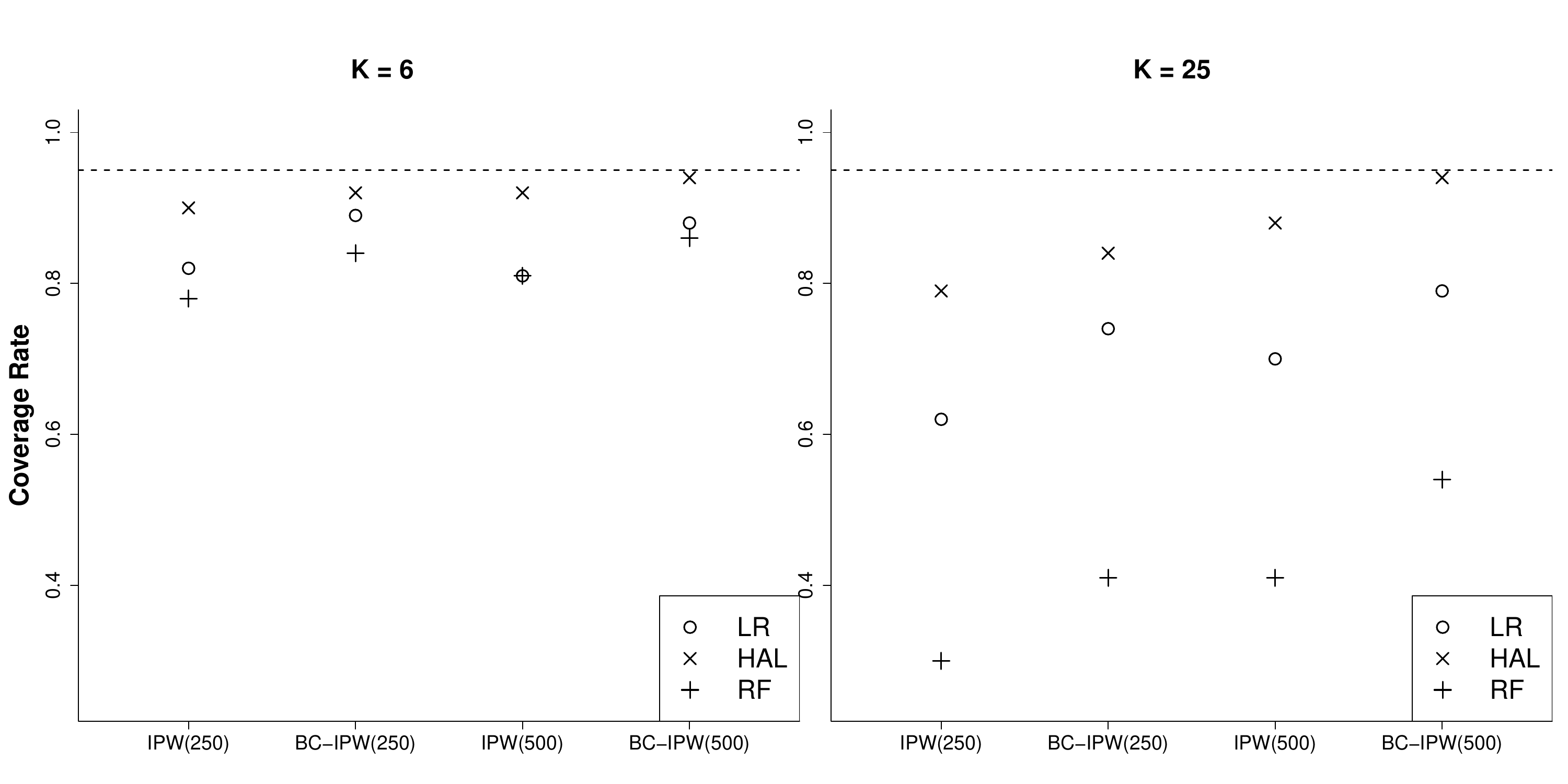}
   \caption{The coverage rates of the estimated RQAL in Scenarios 1 (first row) and 2 (second row) with logistic (LR), highly adaptive lasso (HAL) and random forest (RF) hazard models. IPW is the estimator \eqref{eq:rqaldirect} while BC-IPW is the bias-corrected estimator \eqref{eq:rqalsmooth}. The number in brackets is the sample size and the dashed line shows the nominal rate of 0.95. }
   \label{fg:4.1}
\end{figure}

\begin{table}[htbp!]
\centering
\captionsetup{font=small}
\caption{Summary table of simulation studies in Scenario 2. The target parameter is $R(\eta_{\text{opt}}) = 21.04$ when $K = 6$ and $R(\eta_{\text{opt}}) = 31.74$ when $K = 25$. The numbers in brackets are the empirical standard deviation. The notation $\breve{R}$ and $\breve{\eta}$ denote the estimated quantities corresponding to the methods in each row.}
\label{tb:4.2}
\scalebox{0.8}{
\begin{tabular}{ll|c|c|ccc|c|c|c|c|c}
\hline
Method        & Model   & $K$ & $n$  & $\eta_0$        & $\eta_1$        & $\eta_2$        & $\breve{R}(\breve{\eta}_{\text{opt}})$ & SE   & CP   & $R(\breve{\eta}_{\text{opt}})$ & MR           \\ \hline
\multirow{3}{*}{IPW} 
              & Logit   & \multirow{6}{*}{6}  & \multirow{3}{*}{250} & 1.08(0.19) & 1.03(0.32) & 0.96(0.15) & 21.63(0.57) & 0.59 & 0.82 & 20.41(0.26) & 19.24(7.59) \\
              & HAL     &                     &                      & 1.01(0.18) & 1.07(0.35) & 0.99(0.15) & 21.26(0.53) & 0.55 & 0.92 & 20.69(0.33) & 10.72(10.40) \\
              & RF      &                     &                      & 1.04(0.39) & 1.08(0.76) & 1.00(0.32) & 21.75(0.75) & 0.69 & 0.78 & 20.54(0.52) & 16.20(16.04) \\ \cline{4-12}
\multirow{3}{*}{BC-IPW} & Logit   &                     & \multirow{3}{*}{250} & 1.07(0.19) & 1.06(0.35) & 0.95(0.14) & 21.46(0.56) & 0.59 & 0.89 & 20.44(0.25) & 19.15(7.40) \\
              & HAL     &                     &                      & 1.02(0.19) & 1.06(0.36) & 1.00(0.16) & 20.93(0.52) & 0.56 & 0.97 & 20.70(0.29) & 10.46(9.54)  \\
              & RF      &                     &                      & 1.04(0.46) & 1.10(0.81) & 0.98(0.36) & 21.57(0.75) & 0.69 & 0.84 & 20.51(0.51) & 17.09(16.06) \\ \hline
\multirow{3}{*}{IPW} 
              & Logit   & \multirow{6}{*}{6}  & \multirow{3}{*}{500} & 1.06(0.09) & 0.99(0.16) & 0.96(0.08) & 21.49(0.40) & 0.42 & 0.81 & 20.47(0.15) & 17.60(4.55)  \\
              & HAL     &                     &                      & 1.01(0.09) & 1.02(0.17) & 1.00(0.09) & 21.15(0.38) & 0.39 & 0.94 & 20.84(0.16) & 5.81(5.40)   \\
              & RF      &                     &                      & 1.02(0.18) & 1.05(0.38) & 1.00(0.15) & 21.52(0.56) & 0.53 & 0.81 & 20.77(0.30) & 8.92(9.77)   \\ \cline{4-12}
\multirow{3}{*}{BC-IPW} & Logit   &                     & \multirow{3}{*}{500} & 1.06(0.09) & 1.01(0.16) & 0.96(0.07) & 21.37(0.40) & 0.42 & 0.88 & 20.51(0.14) & 17.28(4.28)  \\
              & HAL     &                     &                      & 1.01(0.09) & 1.01(0.16) & 1.00(0.08) & 20.92(0.37) & 0.40 & 0.95 & 20.84(0.15) & 5.79(4.92)   \\
              & RF      &                     &                      & 1.03(0.20) & 1.05(0.37) & 1.00(0.16) & 21.38(0.56) & 0.53 & 0.86 & 20.75(0.32) & 9.50(10.21)  \\ \hline
\multirow{3}{*}{IPW} 
              & Logit   & \multirow{6}{*}{25} & \multirow{3}{*}{250} & 1.10(0.27) & 1.03(0.40) & 0.96(0.17) & 33.13(1.10) & 1.00 & 0.62 & 28.45(0.86) & 66.84(10.44) \\
              & HAL     &                     &                      & 1.05(0.37) & 1.10(0.65) & 0.99(0.27) & 32.37(0.83) & 0.87 & 0.85 & 30.14(0.33) & 33.09(24.18) \\
              & RF      &                     &                      & 1.13(0.78) & 1.10(0.89) & 1.02(0.28) & 33.75(0.89) & 0.72 & 0.30 & 30.05(1.78) & 34.46(24.43) \\ \cline{4-12}
\multirow{3}{*}{BC-IPW} & Logit   &                     & \multirow{3}{*}{250} & 1.11(0.29) & 1.03(0.43) & 0.96(0.18) & 32.75(1.08) & 1.03 & 0.74 & 28.44(0.72) & 68.00(10.73) \\
              & HAL     &                     &                      & 1.07(0.44) & 1.06(0.58) & 1.01(0.27) & 31.59(0.86) & 0.92 & 0.95 & 30.26(1.06) & 32.53(22.61) \\
              & RF      &                     &                      & 1.11(0.81) & 0.96(1.21) & 1.05(0.37) & 33.46(0.94) & 0.77 & 0.41 & 30.05(1.35) & 36.31(24.60) \\ \hline
\multirow{3}{*}{IPW} 
              & Logit   & \multirow{6}{*}{25} & \multirow{3}{*}{500} & 1.06(0.10) & 1.01(0.17) & 0.96(0.09) & 32.73(0.87) & 0.81 & 0.70 & 28.51(0.44) & 67.37(8.15)  \\
              & HAL     &                     &                      & 1.02(0.15) & 1.03(0.23) & 1.00(0.12) & 32.08(0.72) & 0.69 & 0.86 & 30.90(0.77) & 18.08(17.11) \\
              & RF      &                     &                      & 1.05(0.27) & 1.03(0.42) & 1.02(0.19) & 33.31(0.80) & 0.72 & 0.41 & 30.63(0.94) & 25.48(20.44) \\ \cline{4-12}
\multirow{3}{*}{BC-IPW} & Logit   &                     & \multirow{3}{*}{500} & 1.06(0.09) & 1.01(0.16) & 0.96(0.07) & 32.44(0.88) & 0.84 & 0.79 & 28.48(0.41) & 68.66(7.68)  \\
              & HAL     &                     &                      & 1.01(0.11) & 1.02(0.20) & 1.00(0.10) & 31.47(0.70) & 0.73 & 0.95 & 30.95(0.65) & 17.20(14.85) \\
              & RF      &                     &                      & 1.04(0.28) & 1.08(0.54) & 1.00(0.21) & 33.04(0.83) & 0.76 & 0.54 & 30.56(1.02) & 26.84(22.01) \\ \hline
\end{tabular}}
\end{table}

\subsection{Misspecified Hazard Models}
In Scenario 2, the logistic models include the transformed covariates $X$ instead of $Z$, and thus, logistic regressions using $X$ are misspecified. We consider two setups for $L$ and $G$ that are similar to those in Section~\ref{sim:corect} and generate 500 datasets of sizes 250 and 500.

 Table~\ref{tb:4.2} summarizes the results of Scenario 2. In this table, IPW is the direct estimation using \eqref{eq:rqaldirect}, BC-IPW is the bias-correction estimator \eqref{eq:rqalsmooth}; Logit, logistic regression, HAl, highly adaptive lasso, RF, random forest; $K$ is the number of stages; $n$ is the sample size; $\eta_0$, $\eta_1$, $\eta_2$ are the empirical average of $\breve{\eta}_0/\breve{\eta}_1$, $\breve{\eta}_1/\breve{\eta}_2$, $\breve{\eta}_2/\breve{\eta}_0$; $\breve{R}(\breve{\eta}_{opt})$ is the empirical average of  estimated RQAL of the estimated optimal regime. SE is the standard error estimation; CP is the average coverage probability; $R(\breve{\eta}_{opt})$ is the average of true RQAL of the estimated optimal regime; MR is the empirical mis-classification rate of the estimated optimal regime. The numbers in brackets are the empirical standard deviation. 
 
 The undersmoothed highly adaptive lasso outperforms the logistic regression and random forest based estimators in terms of bias, coverage probabilities and misclassification rates. Specifically, the highly adaptive lasso leads to an  estimator with smallest bias in $\hat{R}(\hat{\eta}_{opt})$, $\tilde{R}(\tilde{\eta}_{opt})$, ${R}(\hat{\eta}_{opt})$ and ${R}(\tilde{\eta}_{opt})$ and the coverage probabilities are close to the nominal rate. The misclasification rate of the highly adaptive lasso is also substantially lower than the other two estimators. Figure \ref{fg:4.1} (second row) displays the coverage probabilities where the dashed line represents the nominal coverage of 0.95. The random forest results in coverage probabilities as low as 0.3 with $n=250$ and $K=25$. The highly adaptive lasso estimator has the coverage probability of 0.95 for the bias corrected estimator for $n=500$ while the coverage probabilities for the logistic regression and the random forest are 0.77 and 0.54, respectively.

Similar to Section~\ref{sim:corect}, we provide additional information about the distribution of the estimators in both K=6 and 25 in  Section A.4 of the Supplementary Material (Figures \ref{fg:s3} and \ref{fg:s4}).  The figures provide insight about the distribution of (a) the true restricted quality-adjusted lifetime under different estimated optimal regimes (i.e., $R(\hat{\eta}_{opt})$ and $R(\tilde{\eta}_{opt})$); (b)  the estimated restricted quality-adjusted lifetime under the estimated optimal regime (i.e., $\hat R(\hat{\eta}_{opt})$ and $\tilde R(\tilde{\eta}_{opt})$); and, (c) the misclassification rates. The plots does not show any evidence of deviating from normal distribution of our estimators, thereby, supporting our theoretical results. The highly adaptive lasso outperforms the logistic regression and the random forest for both sample sized 250 and 500. The importance of using  the highly adaptive lasso is more pronounced when there are more stages (i.e., K=25 in Figures \ref{fg:s4}).

In Section A.5 of the Supplementary Material, we provide additional simulation studies to compare the performance of the undersmoothed and the cross-validated highly adaptive lasso. The results show that while the cross-validated based estimator results in severely under-covered confidence intervals (as low as 69\%), the coverage of the undersmoothed based estimator is close to the nominal rate. 

\section{Application: The effect of PEG tube in patients with ALS } \label{sec:real}

\subsection{Existing results}

\citet{atassi2011advanced}  used marginal structural models to assess the causal effect of PEG on survival time among  patients with ALS \citep{robins2000marginal}. They concluded that PEG placement was associated with a 28\% increase in the hazard of death, permanent assisted ventilation (PAV), or tracheostomy after accounting for confounding from vital capacity and functional capability. \citet{mcdonnell2017causal} also found that PEG placement may harm survival with no quality of life related benefit for people with ALS by considering marginal structural models and structural nested models \citep{van2005multidisciplinary, hernan2005structural}. There are also a number of studies on elderly subjects which  did not find evidence to support PEG feeding \citep{callahan2000outcomes, finucane2007tube}. Despite the lack of convincing evidence on the benefits of PEG tube placement, it is still widely used and considered as the current standard of care \citep{rabeneck1997ethically, rowland2001amyotrophic, miller2009practice, efns2012efns}.

Most of the existing substantive literature has focused on the average effect of PEG on survival time or quality of life among ALS patients. However, in practice,  an alternative clinical endpoint is a combination of the quantity and quality of life \citep{rabeneck1997ethically}. For example, for a subgroup of patients, PEG insertion may increase the survival time but reduce the quality of life drastically; conversely, delaying PEG insertion may improve quality of life but decrease the survival time. This investigation is motivated by a recent paper by Dr. David Schoenfeld and colleagues \citep{mcdonnell2017causal} that concluded PEG insertion has no benefit for survival or quality of life outcomes.  

\subsection{Data source}

We use the data from a double-blinded, placebo-controlled randomized clinical trial
(NCT00349622) designed to assess ceftriaxone's effects on survival in patients with ALS \citep{berry2013design, cudkowicz2014safety}. From 2006 to 2012, the ceftriaxone trial enrolled  513 participants with
laboratory-supported probable, probable, or definite ALS according
to the El Escorial criteria \citep{brooks2000escorial}. Patients visit the clinic every four weeks by protocol, and may receive PEG based on their disease progression and status.
 Therefore, PEG insertion is not randomized and its effect can be confounded by patients' characteristics. 


\subsection{Our approach}

We are using data of 481 samples from a multicenter placebo-controlled RCT (NCT00349622) designed to assess ceftriaxone's effects on survival
and functional decline in patients with ALS. Per protocol, participants completed assessments at baseline; weeks 1, 2, and 4; and every four weeks thereafter. Dates of PEG placement were recorded. To simplify the analysis,  we only consider the first 120 weeks (i.e., at most 30 time points) data, because more than 95\% of patients are dead or censored before this time points. 
Censoring is  defined at the time after more than a 17-week gap between the 2 ALS Functional Rating Scale (ALSFRS-R) assessments or more than a 33-week gap
between 2 BMI, FVC, grip strength, or ALSSQOL measurements \citep{mcdonnell2017causal}.
We set the interval between two assessments as $G = 4$. Thus, we have a $K = 30$ design. Baseline characteristics considered for confounding included sex, age,
disease duration, site of symptom onset, respiratory comorbidities
(history and prevalence at screening), cardiovascular comorbidities
(history and prevalence at screening), use of riluzole, use of
ceftriaxone, ALSFRS-R total
score and four subscales (bulbar, breathing, fine motor, and gross
motor), percent change in body mass index (BMI) from screening,
forced vital capacity (FVC), grip strength using handheld
dynamometry, and the ALSSQOL. Survival was defined as the time in weeks from screening until death, PAV, or tracheostomy. Quality of life (QOL) was measured
using the average total score of the ALS-Specific Quality of Life
(ALSSQOL) inventory, a 59-item patient-reported assessment
that evaluates perceived QOL concerning negative emotion, interaction with people and the environment, intimacy, religiosity, physical symptoms, and bulbar function. The utility coefficient map is defined as $\mathcal{Q} := \mathrm{ALSSQOL(t)}/\max_t \mathrm{ALSSQOL(t)}$.  Here, we focus on maximizing RQAL as the outcome of interest while the restricted value for QAL is set as $L_U = 91.12$ by picking the 95\% quantile of the observed QAL. According to experts' advice, two important time-dependent clinical variables, percent change in BMI from screening and FVC, are considered in the treatment regime $g^\eta$, which means  the PEG insertion is recommended based on the value of the linear combination of percent change of BMI and FVC. 

Table~\ref{app:tb:logit} in Section \ref{app:addtable} of the Supplementary Material shows the estimated coefficients of the measured variables on the probability of PEG tube insertion and censoring on logit scale.  Based on these results, the ALS bulbar and FVC are the most significant variables in deciding for PEG placement such that patients with higher values of bulbar and FVC are less likely to receive PEG. Also, the stage (i.e., number of months since the beginning of the follow-up), site of onset, Ceftriaxone, and ALS total score contribute significantly in hazard of censoring.  

To estimate the optimal treatment length strategy, we apply the proposed method using logistic regression and highly adaptive lasso for the hazard models. The estimate nuisance parameters are plugged into the value function \eqref{eq:rqaldirect} and \eqref{eq:rqalsmooth} and maximized to obtain the estimated optimal strategy. We find that logistic regression and highly adaptive lasso return similar RQAL estimation under different regimes. The only difference is the effect of BMI reduction in the estimated optimal regimes. The results of estimated optimal strategy $g^{\hat{\eta}_{opt}}$ and $g^{\tilde{\eta}_{opt}}$ are summarized in Table~\ref{tb:est-opt}. Since these two estimated regimes are nearly identical, we consider only $g^{\hat{\eta}_{opt}}$ for simplicity of representation. The coefficients for FVC and BMI reduction are negative, which means $g^{\hat{\eta}_{opt}}$ tends to recommend PEG when both FVC and BMI reduction decrease (when BMI decreases, this is a negative number, and thus, the larger is the score, the better is the patient), and the direction of these parameters coincides with the experts' knowledge. The FVC variable is on a large scale, which explains the small coefficient.

\begin{figure}[t]
\centering
\captionsetup{font=small}
  \includegraphics[width=0.4\textwidth]{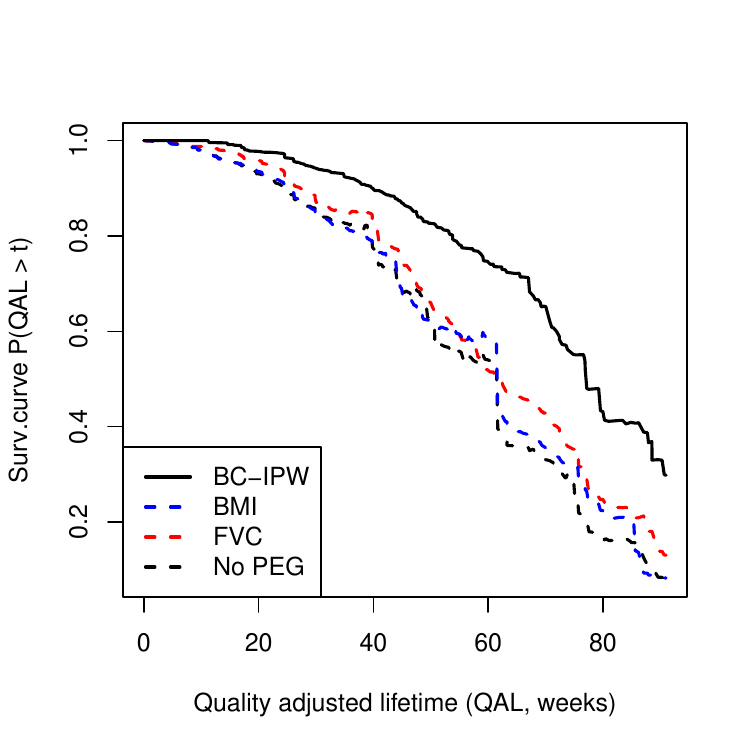}
   \caption{The estimated survival curve of QAL under regime $g^{\tilde{\eta}_{opt}}$ (solid line), never insertion (black dashed line), always insertion (dotted line), FVC regime (red dashed line), and BMI regime (blue dashed line) with logistic hazard models. }
   \label{fg:4.4}
\end{figure}

We  compare the RQAL estimation with four other treatment regimes: immediate PEG insertion ($R(1)$), never PEG insertion ($R(0)$), insert the PEG when FVC score is less than 50 ($R(F)$) and insert the PEG when BMI changes -10\% or greater ($R(B)$).  The latter two treatment regimes are currently used in clinical practice. The estimated regime $g^{\hat{\eta}_{opt}}$ is significantly better than the always/never insertion strategies by looking at the point estimation and 95\% confidence intervals of RQAL.   Moreover, the estimated RQAL under the estimated treatment regime using highly adaptive lasso is roughly 16 and 11 points higher than the RQAL obtained by the existing treatment rules $\hat{R}(B)$ and $\hat{R}(F)$, respectively. The difference between $\hat R ({\hat{\eta}_{opt}})$ and $\hat{R}(B)$ is also statistically significant. Regardless, the estimated regime provides clinically significant improvement on RQAL  compared with both existing rules.

The IPW and Bias-correction estimators are equal for always/never insertion regimes. The reason is because we use $\eta = (1, 0, 0)$ and $(-1, 0, 0)$ as the representative regimes with only intercept, and $sd(\mathcal{Z})$, a scale factor in the bandwidth $h$, would be 0 under these regimes. With 0 bandwidth,  the bias-corrected estimator reduces to the IPW estimator. Another issue is that the standard error estimation of the immediate insertion regime is pretty large. This is expected since few people (2.5\%) in this data set start the treatment at stage 1, and this fact results in the high variability of the proposed IPW estimator. This phenomenon can also be seen from Figure~\ref{fg:4.4} about the overall estimated survival curve.

\begin{table}[H]
\centering
\captionsetup{font=small}
\caption{Estimated optimal treatment length strategy for PEG inssertion.}\label{tb:est-opt}
\scalebox{0.75}{
\begin{tabular}{llllllllll}
\hline
            PS           &   Est    & (Int.) & FVC    & BMI Red. & $\hat{R}(\hat{\eta}_{opt})$           & $\hat{R}(1)$        & $\hat{R}(0)$  &     $\hat{R}(F)$  &     $\hat{R}(B)$ \\ \hline
\multirow{2}{*}{Logit} & IPW   & 0.995       & -0.011 & -0.097        & 72.24(4.75) & 33.07(8.08) & 57.30(2.51) & 61.26(2.57) & 58.51(3.17)\\
                       & BC-IPW & 0.996       & -0.014 & -0.084        & 71.95(4.14) & 33.07(8.08) & 57.30(2.51) & 60.84(2.57) & 58.52(3.17)\\ \hline
\multirow{2}{*}{HAL}   & IPW   & 0.998       & -0.012 & -0.066        & 76.09(4.29) & 23.59(7.08) & 60.24(2.16) & 65.01(2.53) & 60.66(2.39)   \\
                       & BC-IPW & 0.998       & -0.012 & -0.054        & 75.05(4.25) & 23.59(7.08) &  60.24(2.16) & 64.70(2.53) & 60.66(2.39)  \\ \hline
\end{tabular}}
\end{table}

Figure~\ref{fg:4.4} shows the estimated survival curves of QAL under different regimes using logistic regression. For simplicity, we only plot the survival curve of the IPW estimator. The estimated survival curve under the estimated optimal strategy $g^{\hat{\eta}_{opt}}$ is uniformly better than under the other regimes, indicating that $g^{\hat{\eta}_{opt}}$ may lead to improved  QAL if used in ALS study. Additionally, the estimated survival curve  shows that the regime that no patients insert the PEG leads to better patient outcomes than the regime that all patients insert the PEG immediately, which partially explains why the previous studies by using MSM conclude that PEG insertion is harmful to survival for patients with ALS.  

\section{Discussion} \label{sec:dis}

  There are several directions for future work with this interesting problem. The proposed methods directly maximize the estimated restricted quality-adjusted lifetime using a genetic algorithm and it is computationally expensive in multi-stage cases. It is also challenging to handle high-dimensional covariate space in $g^{\eta}$. The extensions of the proposed method to allow for high-dimensional covariates could be considered. Also, the proposed method only search optimal strategies in a restricted functional class, say the linear classifiers with the linear combination of some covariates. This type of regimes are easy to interpret but can be sub-optimal when the actual optimal regime is complex. Thus, the trade-off between interpretability and accuracy of the estimated optimal regime and extending the search space of optimal regime to more general functional class would be considered.  

\bibliographystyle{rss}
\bibliography{sample}

\pagebreak


\appendix

\setcounter{section}{0}
\renewcommand*{\thesection}{S\arabic{section}}
\renewcommand*{\thesubsection}{S\arabic{section}.\arabic{subsection}}
\renewcommand{\theequation}{S\arabic{equation}}
\renewcommand{\thefigure}{S\arabic{figure}}
\renewcommand{\thetable}{S\arabic{table}}
\renewcommand{\thetheorem}{S\arabic{theorem}}
\renewcommand{\thelemma}{S\arabic{lemma}}
\renewcommand{\bibnumfmt}[1]{[S#1]}
\renewcommand{\citenumfont}[1]{S#1}
\setcounter{page}{1}

\section{Appendix} \label{app-ch3}
We first list the regularity conditions needed for this section.
\par
\medskip
{\bf Regularity Conditions:}
\begin{enumerate}
    \item[(C.1)]  The covariates $Z_j$ are bounded. 
    \item[(C.2)]  $H_a(x;\bar{z})$ is bounded away from 0 and 1 for all possible values of $x$ and $\bar{z}$; $H_c(x; \bar{a}, \bar{z})$ is bounded away from 0 and 1 for all possible values of $(x, \bar{a}, \bar{z})$.
    \item[(C.3)] The treatment assignment model and censoring model are both correctly specified by logistic regression models, e.g., $P(C_j = 1 \mid \bar{Z}_{j}, C_{j-1} = 0) = 1/\left\{1 + \exp(\beta_j^{\top}\bar Z_j)\right\}$ and $P(A_j = 1\mid \bar{Z}_{j}, A_{j-1} = 0) = 1/\left\{1 + \exp(\theta_j^{\top}\bar Z_j)\right\}$. 
    \item[(C.4)] $|\Delta_{c}(x) / H_a(x, \bar{Z})H_c(x, \bar{A}, \bar{Z})| < \infty$, $|\Delta_{c}(x)I(U > x)/ H_a(x, \bar{Z})H_c(x, \bar{A}, \bar{Z})| < \infty$ uniformly for $x \in (0, L_U]$ and $\eta$.
    \item[(C.5)] $n\nu \rightarrow \infty$ and $n\nu^4 \rightarrow 0$ as $n \rightarrow \infty$.
\end{enumerate}

Condition~(C.1) is a common assumption of convenience used to bound the score function and influence curves.  Condition~(C.2) is required to bound the weights $H_a$ and $H_c$ away from zero such that their reciprocals are finite.
Condition~(C.3) implies that, for $j=1,\cdots,K$, $\hat\beta_j$ and $\hat\theta_j$ are consistent and asymptotic normal at the nominal $\sqrt{n}$-rate, i.e., $n^{1/2}(\hat{\beta}_j - \beta_j) = n^{1/2} \mathbb{P}_n \{IC_{j,\beta}(O)\} + o_p(1)$ and $n^{1/2}(\hat{\theta}_j - \theta_j) = n^{1/2} \mathbb{P}_n \{ IC_{j,\theta}(O)\} + o_p(1)$. Condition~(C.5) gives common conditions on the smoothing parameter 

\subsection{Proof of Theorem~\ref{ch3-theorem1}}
\begin{proof}

In this proof, we omit the superscript '$^{par}$' for simplicity, with the understanding that we are focusing on the parametric modeling of the nuisance functions.

For simplicity, we suppress the covariate information in $H_c(x; \bar{a}, \bar{z})$ and $H_a(x; \bar{z})$, and use $H_a(x)$ and $H_c(x)$ instead. We divide the proof into three steps.
 
\subsubsection*{Step 1: The consistency and asymptotic normality of $\hat{S}_U(x; \eta)$}
We first show the consistency of the direct IPW estimator. Fix a target time $x$ and note that $\Delta^\eta_{a}(x)I(U > x) = \Delta^\eta_{a}(x)I(U^{\eta} > x)$ under Assumption \ref{assump:basic}a.
Thus, $E[\Delta^\eta_{a}(x)I(U > x)] = S_U(x; \eta)$. 
Then, by an ordinary Weak Law of Large Numbers,
\[
\mathbb{P}_n \left[ \frac{\Delta^\eta_{a}(x)\Delta_{c}(x)I(U > x)}{\hat{H}_{a}(x)\hat{H}_{c}(x)} \right]
    \rightarrow_p   E\left[ \frac{\Delta^\eta_{a}(x)\Delta_{c}(x)I(U > x)}{H_{a}(x)H_{c}(x)} \right] \\
= S_U(x;\eta),
\]
 uniformly for $x \in (0, L_U]$ as $n \rightarrow \infty$ because of the correct specification of parametric propensity score models in Condition~(C.2). Similarly, we have $\mathbb{P}_n
    [\Delta^\eta_{a}(x)\Delta_{c}(x)/\hat{H}_{a}(x)\hat{H}_{c}(x)] \rightarrow_p E[\Delta^\eta_{a}(x)\Delta_{c}(x)/ H_{a}(x)H_{c}(x)] = 1$ uniformly for $x \in (0, L_U]$ as $n \rightarrow \infty$.  
Putting these two results together leads one to conclude that 
$\hat{S}_U(x;\eta)\rightarrow_p S_U(x;\eta)$ as $n\rightarrow\infty$.

We then show the asymptotic normality of $n^{1/2}\left\{\hat{S}_U(x;\eta) - S_U(x;\eta)\right\}$. Let $\hat{S}_U(x; \eta, \theta, \beta)$ denote an estimator where the nuisance functions are known.  Based on \eqref{eq:direct} and first order Taylor expansion, we have
\begin{align} \label{eq:asy prob} \tag{A.2}
    \begin{split}
        & n^{1/2}\{\hat{S}_U(x;\eta) - S_U(x;\eta)\} \\
        =& n^{1/2}\{\hat{S}_U(x; \eta, \theta, \beta) - S_U(x;\eta)\} + D_1^{\top}(x) n^{1/2}(\hat{\theta} - \theta) + D_2^{\top}(x) n^{1/2}(\hat{\beta} - \beta) + o_p(1)\\
        =& n^{-1/2} \frac{\sum_{i=1}^{n}   \frac{\Delta^\eta_{a,i}(x)\Delta_{c,i}(x)}{H_{a, i}(x;\theta)H_{c, i}(x; \beta)}I(U_i > x)}{1/n\sum_{i=1}^{n}   \frac{\Delta^\eta_{a,i}(x)\Delta_{c,i}(x)}{H_{a, i}(x;\theta)H_{c, i}(x; \beta)}} - S_U(x;\eta) \\
        &+ D_1^{\top}(x) n^{1/2}(\hat{\theta} - \theta) + D_2^{\top}(x) n^{1/2}(\hat{\beta} - \beta) + o_p(1) \\
        =& n^{-1/2} \frac{\sum_{i=1}^{n}   \left\{\frac{\Delta^\eta_{a,i}(x)\Delta_{c,i}(x)}{H_{a, i}(x;\theta)H_{c, i}(x; \beta)}I(U_i > x) -   \frac{\Delta^\eta_{a,i}(x)\Delta_{c,i}(x)}{H_{a, i}(x;\theta)H_{c, i}(x; \beta)}S_U(x;\eta)\right\}}{1/n\sum_{i=1}^{n}   \frac{\Delta^\eta_{a,i}(x)\Delta_{c,i}(x)}{H_{a, i}(x;\theta)H_{c, i}(x; \beta)}} \\
        &+ D_1^{\top}(x) n^{1/2}(\hat{\theta} - \theta) + D_2^{\top}(x) n^{1/2}(\hat{\beta} - \beta) + o_p(1) \\
        =& n^{-1/2} \sum_{i=1}^{n}  \frac{\Delta^\eta_{a,i}(x)\Delta_{c,i}(x)}{H_{a, i}(x;\theta)H_{c, i}(x; \beta)}\left\{I(U_i > x) - S_U(x;\eta)\right\}\\
        &+ D_1^{\top}(x) n^{1/2}(\hat{\theta} - \theta) + D_2^{\top}(x) n^{1/2}(\hat{\beta} - \beta) + o_p(1) \\
        =& n^{-1/2} \sum_{i=1}^{n}  \{IC_{x;\eta}(o_i) + D_1^{\top}(x) IC_{\theta}(o_i) + D_2^{\top}(x)IC_{\beta}(o_i)\}+ o_p(1)\\
        := & n^{-1/2} \sum_{i=1}^{n} IC_{PAR}(o_i; x, \eta) + o_p(1),
    \end{split}
\end{align}
where $D_1(x) = \lim_{n \rightarrow \infty} \partial \hat{S}_U(x;\eta)/\partial \theta$, $D_2(x) = \lim_{n \rightarrow \infty} \partial \hat{S}_U(x;\eta)/\partial \beta$, $IC_{x;\eta}(o)$ is the influence function under regime $g^{\eta}$ at time $x$ when all nuisnace parameters are known, and $IC_{\beta}(o)$ and $IC_{\theta}(o)$ are part of the parametric models defined in regularity conditions (C.3). $\hat{S}_U(x;\eta) - S_U(x;\eta)$ is thus an asymptotic linear estimator within the range $(0, L_U]$. The first term is composed of bounded probability functions and simple indicator functions, which belong to the Donsker class, and thus $n^{-1/2}\left\{\hat{S}_U(x;\eta) - S_U(x;\eta)\right\}$ converges to a mean zero Gaussian process by empirical process theory \cite{van2000asymptotic}. 

Consequently, a consistent covariance estimation is given by 
$1/n \sum_{i=1}^{n} \widehat{IC}_{PAR}(o_i; x, \eta)^2$ with estimated $\hat{\theta}$ and $\hat{\beta}$, where
\[\widehat{IC}_{PAR}(o_i; x, \eta) = \frac{ \left\{\frac{\Delta^\eta_{a,i}(x)\Delta_{c,i}(x)}{H_{a, i}(x;\theta)H_{c, i}(x; \beta)}I(U_i > x) - \hat{S}_U(x;\eta)\right\}}{1/n\sum_{i=1}^{n}   \frac{\Delta^\eta_{a,i}(x)\Delta_{c,i}(x)}{H_{a, i}(x;\theta)H_{c, i}(x; \beta)}} + \hat{D}_1^{\top}(x)\widehat{IC}_{\theta}(o_i) + \hat{D}_2^{\top}(x)\widehat{IC}_{\beta}(o_i).\]

\subsubsection*{Step 2: The asymptotic distribution of $\hat{S}_U(x; \hat{\eta}_{opt})$}
Let $\hat{\eta}_{opt}$ be the maximizer of $\hat{S}_U(x;\eta)$ and $\eta_{opt}$ be the maximizer of $S_U(x;\eta)$. We assume $S_U(x;\eta)$ is a smooth function of $\eta$. Then by the consistency of $\hat{S}_U(x;\eta)$, $\hat\eta_{opt}$ will converge in probability to $\eta_{opt}$. Let $W_n(\eta) = n^{1/2}\left\{\hat{S}_U(x;\eta) - S_U(x;\eta)\right\} $. Following the argument in \citet{zhang2012robust} and using a stochastic equicontinuity argument, we have  $W_n(\hat{\eta}_{opt}) - W_n(\eta_{opt})
\rightarrow 0$ in probability. We can represent
\begin{align} \tag{A.3}
    \begin{split}
        & n^{1/2}\{\hat{S}_U(x; \hat{\eta}_{opt}) - S_U(x; \eta_{opt})\} \\
        =& W_n(\hat{\eta}_{opt}) + n^{1/2}\{S_U(x; \hat{\eta}_{opt}) - S_U(x; \eta_{opt}) \} \\
        =& W_n(\hat{\eta}_{opt}) - W_n(\eta_{opt}) +  W_n(\eta_{opt})  + n^{1/2}\{S_U(x; \hat{\eta}_{opt}) - S_U(x; \eta_{opt}) \} \\
        =& W_n(\hat{\eta}_{opt}) - W_n(\eta_{opt}) +  W_n(\eta_{opt})  + n^{1/2}(\hat{\eta}_{opt} - \eta_{opt} )\frac{\partial}{\partial \eta}S_U(x; \eta)\mid_{\eta_{opt}} + o_p(1)\\
        =& W_n(\eta_{opt}) + o_p(1),
    \end{split}
\end{align}
The last equality follows because $W_n(\hat{\eta}_{opt}) = o_p(1)$ and $\frac{\partial}{\partial \eta}S_U(x; \eta)\mid_{\eta_{opt}}=0$ ($\eta_{opt}$ is the maximizer of the function). Thus, $n^{1/2}\left\{\hat{S}_U(x; \hat{\eta}_{opt}) - S_U(x; \eta_{opt})\right\}$ converges to the same Gaussian process as $n^{1/2}\left\{\hat{S}_U(x; \eta_{opt}) - S_U(x; \eta_{opt})\right\}$. That is $\hat{S}_U(x; \hat{\eta}_{opt})$ converges to a Gaussian process with mean $S_U(x; \eta_{opt})$ and variance $E[IC_{PAR}(o;x, \eta_{opt})^2]$ which can be empirically estimated as
\begin{equation} \tag{A.4}
    1/n\sum_{i=1}^{n} \widehat{IC}_{PAR}(o_i;x, \hat{\eta}_{opt})^2.
\end{equation}

\subsubsection*{Step 3: The asymptotic equivalence between $\hat{S}_U(x; \hat{\eta}_{opt})$ and $\tilde{S}_U(x; \tilde{\eta}_{opt})$}
Define $Q = E\left\{\Delta^{\eta}_{a,i}(x)\Delta_{c,i}(x)/H_{a, i}(x)H_{c, i}(x)\right\}$ and $\tilde Q = E\left\{\tilde \Delta_{a, i}^\eta(x)\Delta_{c,i}(x)/H_{a, i}(x)H_{c, i}(x)\right\}$. $K$ is actually 1. Let's first consider the case with one stage.

\subsubsection*{K = 1}

Let's look at the trivial case with stage K = 1. We first show that $n^{1/2}(\tilde Q  - Q) = o_p(1)$. Notice that 
\begin{align} \label{eq:s1k1} \tag{A.5}
    \begin{split}
        & n^{-1/2}\sum_{i=1}^{n}\frac{ \{\Delta^{\eta}_{a,i}(x) - \tilde \Delta_{a, i}^\eta(x)\}\Delta_{c,i}(x)}{\hat{H}_{a, i}(x)\hat{H}_{c, i}(x)} \\
        =&n^{-1/2} \sum_{i=1}^{n}\frac{ \{\Delta^{\eta}_{a,i}(x) - \tilde \Delta_{a, i}^\eta(x)\}\Delta_{c,i}(x)}{H_{a, i}(x)H_{c, i}(x)} + o_p(1)
    \end{split}
\end{align}
The definition of $\Delta^{\eta}_{a}(x)$ and $\tilde \Delta_{a}^\eta(x)$ tell us that when K= 1, they are the indicator and normal distribution function themselves. Define $L = ( A_0, T, \delta)$, $\Xi(Z_0,L) = \Delta_{c}(x) / H_{a}(x)H_{c}(x)$ and $r^{\eta} = \eta^{\top} Z_0$. Following the arguments in \citet{heller2007smoothed}, we have
 \begin{align} \label{eq:s1k2} \tag{A.6}
     \begin{split}
         |\eqref{eq:s1k1}| \le M_1 n^{1/2} \sup_{|\eta| = 1}\left|\int_{l} \int_{r^{\eta}} \{I(r^{\eta} \ge 0) - \Phi(r^{\eta} \ge 0)\}
         {\Xi}(l) d\hat{F}(r^{\eta}|l) d\hat{G}(l)\right|
     \end{split}
 \end{align}
 where $M_1$ is a finite constant, $\hat{G}(l)$ and $\hat{F}(r^{\eta}|l)$ are the marginal empirical cumulative distribution functions for $l$ and the conditional empirical cumulative distribution function for $r^{\eta}$, respectively. For simplicity, we omit the superscript $\eta$ in $r^{\eta}$. \eqref{eq:s1k1} is thus bounded by $M_1 n^{1/2} \sup_{\|\eta\| = 1} U$. Write $U$ as $U_1 + U_2$ where 
 \begin{align*} \tag{A.7}
     \begin{split}
         U_1 =& \int_{l} \int_{r^{\eta}} \{I(r^{\eta} \ge 0) - \Phi(r^{\eta} \ge 0)\}
         {\Xi}(l) \{d\hat{F}(r^{\eta}|q) - dF(r^{\eta}|q)\} d\hat{G}(q),\\
         U_2 =& \int_{l} \int_{r^{\eta}} \{I(r^{\eta} \ge 0) - \Phi(r^{\eta} \ge 0)\}
         {\Xi}(l) dF(r^{\eta}|l) d\hat{G}(l),
     \end{split}
 \end{align*}
 with $F(r^{\eta}|l) = \lim_{n \rightarrow \infty} \hat{F}(r^{\eta}|l)$. By variable transformation $z = r^{\eta}/\nu$ and integration by parts, we have 
 \begin{equation} \label{eq:s1k3} \tag{A.8}
     U_1 = \int_{l} \int_{z} {\Xi}(l) \phi(z) \left\{\left[\hat{F}(z\nu|l) - F(z\nu|l)\right] - \left[\hat{F}(0|l) - F(0|l)\right] \right\}dz d\hat{G}(q),
 \end{equation}
 where $\phi(z)$ is the probability density function of standard normal distribution. Under regularity condition (C.4), we apply the results on oscillations of empirical process (Shorack and Wellner, 2009, Theorem 1, p. 542) to equation \eqref{eq:s1k3} and have
 \begin{equation*}
     n^{1/2} U_1 = O_p\left( \sqrt{\nu\log n \log\left(\frac{1}{\nu \log n}\right)}\right).
 \end{equation*}
 In addition, by similar arguments and applying second order Taylor expansion of $U_2$ with respect to $\nu$ around 0, we have
 \begin{equation*}
     U_2 = -\nu^2/2 \int_{q} \int_{z} {\Xi}(l) \phi(z) f^{\prime}(z\nu^{\ast}|l) z^2dz d\hat{G}(l),
 \end{equation*}
where $f^{\prime}(z\nu^{\ast}|l) = \partial^2 F(u|l)/ \partial u^2$ and $\nu^{\ast}$ lies between $\nu$ and 0. We thus have $n^{1/2} U_2 = O_p(n^{1/2}\nu^2)$. Combine the above results, we have 
\begin{equation*}
      M_1 n^{1/2} U \le M_1 n^{1/2} (|U_1| + |U_2|) = O_p\left( \sqrt{\nu\log n \log\left(\frac{1}{ \nu \log n}\right)} + n^{1/2}\nu^2\right).
\end{equation*}
By condition (C.4)-(C.5), we have $\sup_{|\eta| = 1}M_1 n^{1/2} U  = o_p(1)$, which means $n^{1/2}(\tilde Q - Q) = o_p(1)$.

With this result, let's look at the next step:
\begin{align} \label{eq:s1k4} \tag{A.9}
    \begin{split}
        & n^{1/2}\left\{\hat{S}_U(x;\eta) - \tilde{S}_U(x;\eta)\right\}\\
        =& n^{-1/2} \frac{\sum_{i=1}^{n}\frac{ \Delta^{\eta}_{a,i}(x)\Delta_{c,i}(x)}{\hat{H}_{a, i}(x)\hat{H}_{c, i}(x)}I(U_i > x)}{1/n\sum_{i=1}^{n}\frac{ \Delta^{\eta}_{a,i}(x)\Delta_{c,i}(x)}{\hat{H}_{a, i}(x)\hat{H}_{c, i}(x)}} - \frac{\sum_{i=1}^{n}\frac{ \tilde \Delta_{a, i}^\eta(x)\Delta_{c,i}(x)}{\hat{H}_{a, i}(x)\hat{H}_{c, i}(x)}I(U_i > x)}{1/n\sum_{i=1}^{n}\frac{ \tilde \Delta_{a, i}^\eta(x)\Delta_{c,i}(x)}{\hat{H}_{a, i}(x)\hat{H}_{c, i}(x)}}\\
         =& \frac{1}{n^{1/2}\tilde Q Q}\left[\tilde Q \sum_{i=1}^{n}\frac{ \{\Delta^{\eta}_{a,i}(x) - \tilde \Delta_{a, i}^\eta(x)\}\Delta_{c,i}(x)I(U_i > x)}{H_{a, i}(x)H_{c, i}(x)}  \right]\\
        &- \frac{1}{\tilde Q Q}\left[(\tilde Q-Q) n^{-1/2}\sum_{i=1}^{n}\frac{\tilde \Delta_{a, i}^\eta(x)\Delta_{c,i}(x)I(U_i > x)}{H_{a, i}(x)H_{c, i}(x)}  \right] + o_p(1). 
    \end{split}
\end{align}
The second term of \eqref{eq:s1k4} is $o_p(1)$ because $n^{1/2}(Q - \tilde Q) = o_p(1)$ and condition (C.4)-(C.5). As for the first term, we can show it is $o_p(1)$ by using similar arguments proving $n^{1/2}(Q - \tilde Q) = o_p(1)$, and we omit it here. Therefore, we have $n^{1/2}\left\{\hat{S}_U(x; \hat{\eta}_{opt}) - \tilde{S}_U(x; \tilde{\eta}_{opt}) \right\} = o_p(1)$ uniformly for $\eta$. 

\subsubsection*{K = 2}
When $K = 2$, $\Delta^{\eta}_{a}(x)$ and $\tilde \Delta_{a}^\eta(x)$ can be either the indicator (distribution) function or a product of two indicator (distribution) functions depend on the map of QAL. Let's also start from the equivalence between $R$ and $Q$. Suppose at target time $x$, there are two group of patients who are at stage 1 or stage 2
\begin{align} \label{eq:s2k1} \tag{A.10}
    \begin{split}
        & n^{-1/2}\sum_{i=1}^{n}\frac{ \{\Delta^{\eta}_{a,i}(x) - \tilde \Delta_{a, i}^\eta(x)\}\Delta_{c,i}(x)}{\hat{H}_{a, i}(x)\hat{H}_{c, i}(x)} \\
        =&n^{-1/2} \sum_{i \in Stage 1}\frac{ \{\Delta^{\eta}_{a,i}(x) - \tilde \Delta_{a, i}^\eta(x)\}\Delta_{c,i}(x)}{H_{a, i}(x)H_{c, i}(x)} + \sum_{i \in Stage 2}\frac{ \{\Delta^{\eta}_{a,i}(x) - \tilde \Delta_{a, i}^\eta(x)\}\Delta_{c,i}(x)}{H_{a, i}(x)H_{c, i}(x)} + o_p(1)\\
        =& n^{-1/2} \sum_{i=1}^{n}\frac{ \{\Delta^{\eta}_{a,i}(x) - \tilde \Delta_{a, i}^\eta(x)\}\Delta_{c,i}(x)I\{L_i(x) = 1\}}{H_{a, i}(x)H_{c, i}(x)} \\
&+ \sum_{i=1}^{n}\frac{ \{\Delta^{\eta}_{a,i}(x) - \tilde \Delta_{a, i}^\eta(x)\}\Delta_{c,i}(x) I\{L_i(x) = 2\}}{H_{a, i}(x)H_{c, i}(x)} + o_p(1) \\
        := & \mathcal{K}_1 + \mathcal{K}_2 + o_p(1),
    \end{split}
\end{align}
where $L_i(x)$ is the number of stages the $i$th patient stays at target time $x$.  
Notice that $|\eqref{eq:s2k1}| \le = |\mathcal K_1| + |\mathcal K_2|$.
We can use similar arguments in the case $K = 1$ to bound the term $|\mathcal K_1|$ and we omit the details here. Define $L_1 = ( A_0, T_0, \delta_0)$, $L_2 = (L_1,Z_0, A_1, T_1, \delta_1, S)$, $L = (L_1, L_2)$, $K_2(q) = \Delta_{c}(x)I\{K(x) = 2\} / H_{a}(x)H_{c}(x)$ and $r_1^{\eta} = \eta^{\top} Z_0$, $r_2^{\eta} = \eta^{\top} Z_1$.  Following the arguments in \citet{heller2007smoothed}, we have
 \begin{align} \label{eq:s2k2} \tag{A.11}
     \begin{split}
        |\mathcal K_2| \le & M_2 n^{1/2} \\
        & \times  \sup_{|\eta| = 1} \Bigg| \Bigg. \int_{l_1} \int_{r_1^{\eta}} \int_{l_2} \int_{r_2^{\eta}} \{I(r_1^{\eta} \ge 0)I(r_2^{\eta} \ge 0) - \Phi(r_1^{\eta}/\nu)\Phi(r_2^{\eta}/\nu)\}K_2(q) \\
&\times d\hat{F}(r_2^{\eta}|l_2) d\hat{G}(l_2|l_1)d\hat{F}(r_1^{\eta}|l_1) d\hat{G}(l_1) \Bigg. \Bigg|
     \end{split}
 \end{align}
 where $M_2$ is a finite constant, $\hat{G}(l_1)$ and $\hat{F}(r_1^{\eta}|l_1)$ are the marginal empirical cumulative distribution functions for $l_1$ and the conditional empirical cumulative distribution function for $r_1^{\eta}$, respectively at Stage 1 for subjects who belong to Stage 2, while $\hat{F}(r_2^{\eta}|l_2)$ and $\hat{G}(l_2|l_1)$ are the conditional empirical cumulative distribution function for $r_2^{\eta}$ and $l_2$ at Stage 2 for subjects who stay at stage 2. 
 
 Write the absolute integral of the right hand side of \eqref{eq:s2k2} as $U_1 + U_2$ where 
 \begin{align*}
     \begin{split}
        U_1 =& \int_{q_1} \int_{r_1^{\eta}} \int_{l_2} \int_{r_2^{\eta}} \{I(r_1^{\eta} \ge 0)I(r_2^{\eta} \ge 0) - I(r_1^{\eta} \ge 0)\Phi(r_2^{\eta}/\nu)\}
         K_2(l) \\
& \times   d\hat{F}(r_2^{\eta}|l_2) d\hat{G}(l_2|l_1)d\hat{F}(r_1^{\eta}|l_1) d\hat{G}(l_1)\\
         =& \int_{l_1} \int_{r_1^{\eta}} I(r_1^{\eta} \ge 0) \int_{l_2} \int_{r_2^{\eta}} \{I(r_2^{\eta} \ge 0) - \Phi(r_2^{\eta}/\nu)\} \\
& \times
         K_2(q) d\hat{F}(r_2^{\eta}|l_2) d\hat{G}(l_2|l_1) d\hat{F}(r_1^{\eta}|l_1) d\hat{G}(l_1)\\
         =& \int_{l_1} \int_{r_1^{\eta}} I(r_1^{\eta} \ge 0) U_3 d\hat{F}(r_1^{\eta}|l_1) d\hat{G}(l_1),\\
         U_2 =& \int_{l_1} \int_{r_1^{\eta}} \int_{l_2} \int_{r_2^{\eta}} \{I(r_1^{\eta} \ge 0)\Phi(r_2^{\eta}/\nu) - \Phi(r_1^{\eta}/h)\Phi(r_2^{\eta}/\nu)\}\\
& \times
         K_2(q) d\hat{F}(r_2^{\eta}|l_2) d\hat{G}(l_2|l_1)d\hat{F}(r_1^{\eta}|l_1) d\hat{G}(l_1)\\
         =& \int_{l_1} \int_{r_1^{\eta}} \{I(r_1^{\eta} \ge 0) - \Phi(r_1^{\eta}/\nu)\} \int_{l_2} \int_{r_2^{\eta}} \Phi(r_2^{\eta}/\nu) K_2(l) \\
& \times d\hat{F}(r_2^{\eta}|l_2) d\hat{G}(l_2|l_1) d\hat{F}(r_1^{\eta}|q_1) d\hat{G}(l_1).
     \end{split}
 \end{align*}
where $U_3 = \int_{l_2} \int_{r_2^{\eta}} \{I(r_2^{\eta} \ge 0) - \Phi(r_2^{\eta}/\nu)\} 
         K_2(q) d\hat{F}(r_2^{\eta}|l_2) d\hat{G}(l_2|l_1)$. Using the arguments made in the case $K = 1$, Conditions (C.4)-(C.5), and the boundedness of the term
 $\int_{l_2} \int_{r_2^{\eta}} \Phi(r_2^{\eta}/\nu) K_2(l) d\hat{F}(r_2^{\eta}|l_2) d\hat{G}(l_2|l_1)$, we can show that $n^{1/2}U_2 = o_p(1)$ and $n^{1/2}U_3 = o_p(1)$. Because $r_1^{\eta}$ is bounded, $\int_{l_1} \int_{r_1^{\eta}} I(r_1^{\eta} \ge 0) U_3 d\hat{F}(r_1^{\eta}|l_1) d\hat{G}(l_1)$ is also $o_p(1)$. We thus have $n^{1/2}(R - Q) = o_p(1)$. 

We then do the expansion in equation \eqref{eq:s1k4}. The second term is obviously $o_p(1)$ due to $n^{1/2}(R - Q) = o_p(1)$. The first term can also be shown to be $o_p(1)$ by repeating the above arguments when $K = 2$. With fixed finite number of stages $K$, we can do similar things to generalize this result.

Above arguments show that
$n^{1/2}\left\{\hat{S}_U(x;\eta) - \tilde{S}_U(x;\eta)\right\} = o_p(1)$ uniformly in $\eta$. Then by Taylor expansion, $n^{1/2}\left\{\tilde{S}_U(x; \eta_{opt}) - \tilde{S}_U(x; \tilde{\eta}_{opt})\right\} = o_p(1)$ and \\
$n^{1/2}\left\{\hat{S}_U(x; \eta_{opt}) - \hat{S}_U(x; \hat{\eta}_{opt})\right\} = o_p(1)$ and thus $n^{1/2}\left\{\hat{S}_U(x;\hat{\eta}_{opt}) - \tilde{S}_U(x; \tilde{\eta}_{opt})\right\} = \\ n^{1/2}\left\{\hat{S}_U(x; \hat{\eta}_{opt}) - \hat{S}_U(x; \eta_{opt}) + \hat{S}_U(x; \eta_{opt}) - \tilde{S}_U(x; \eta_{opt}) + \tilde{S}_U(x; \eta_{opt})  - \tilde{S}_U(x; \tilde{\eta}_{opt})\right\} = o_p(1)$ and the proof of property 4 in Theorem~\ref{ch3-theorem1} completes.

\end{proof}

\subsection{Proof of Theorem~\ref{ch3-theorem2}}
\begin{proof}

In this proof, we omit the superscript '$^{par}$' for simplicity, with the understanding that we are focusing on the parametric modeling of the nuisance functions.

As shown in Theorem~\ref{ch3-theorem1}, $\hat{S}_U(x; \hat{\eta}_{opt})$ is consistent to $S_U(x; \eta_{opt})$ and 
\begin{equation*}
    n^{1/2} \left\{\hat{S}_U(x; \hat{\eta}_{opt}) - S_U(x; \eta_{opt})\right\} = n^{-1/2} \sum_{i=1}^{n} IC_{PAR}(o_i; x, \eta_{opt}) + o_p(1),
\end{equation*}
where $IC_{PAR}(o; x, \eta_{opt})$ is the influence function when target time is $x$. For the pre-defined restricted upper limit $L_U$, the RQAL $R(\eta)$ up to $L_U$ is a continuous mapping of $S(x; \eta)$. By applying continuous mapping theorem, $\hat{R}(\hat{\eta}_{opt}, \hat{\theta}, \hat{\beta})$ is consistent to $R(\eta_{opt})$ when the model for treatment assignment and censoring are correctly specified.

We then show the asymptotic normality. We already have the asymptotic lineality expression of \eqref{eq:asy prob}. With delta method, it is easy to show that 
\begin{align*}
\begin{split}
     n^{1/2}\left\{\hat{R}(\hat{\eta}_{opt},\hat{\theta}, \hat{\beta}) - R(\eta_{opt})\right\} =& n^{-1/2} \sum_{i=1}^{n} \int_{0}^{L_U} IC_{PAR}(o_i; x, \eta_{opt}) dx + o_p(1) \\
     :=&  n^{-1/2} \sum_{i=1}^{n} IC_{2, PAR}(o_i; x, \eta_{opt}) + o_p(1),
\end{split}
\end{align*}
and $IC_{2, PAR}(o; x, \eta_{opt})$ is the influence function for RQAL at time $x$ under regime $g^{\eta}$.

Thus, $\hat{R}(\hat{\eta}_{opt},\hat{\theta}, \hat{\beta})$ is asymptotic normal with mean $R(\eta_{opt})$ and variance $E[IC_{2, PAR}(o; x, \eta_{opt})^2]$ which can be empirically estimated as
\begin{equation} \tag{A.12}
    1/n\sum_{i=1}^{n} \hat{IC}_{2, PAR}(o_i; x, \hat{\eta}_{opt})^2 = 1/n\sum_{i=1}^{n} \left\{\int_{0}^{L_U} \hat{IC}_{PAR}(o_i; x, \hat{\eta}_{opt})\right\}^2.
\end{equation}
The asymptotic equivalence between $\hat{S}_U(x; \hat{\eta}_{opt})$ and $\tilde{S}_U(x; \tilde{\eta}_{opt})$ can be shown using the arguments for Theorem~\ref{ch3-theorem1}. 
\end{proof}

\subsection{Proof of Theorem~\ref{ch3-theorem3}}
\begin{proof}
In this proof, we omit the superscript '$^{npar}$' for simplicity, with the understanding that we are focusing on the nonparametric modeling of the nuisance functions.

We use the following Lemma to prove the theorem.  Recall
\begin{align*}
    h_{a,j}(g_j^\eta\mid\bar{z}) &= (1-A_{j-1})p(A_j = g_j^\eta \mid A_{j-1} = 0, C_{j-1} = 0, Y_j = 1, \bar{Z}_j = \bar{z})+A_{j-1}\\
    h_{c,j}(c\mid\bar{a}, \bar{z}) &= (1-C_{j-1})p(C_j = c \mid \bar{A}_{j} = \bar{a}, C_{j-1} = 0, Y_j = 1, \bar{Z}_j = \bar{z})+C_{j-1}.
\end{align*}
\begin{lemma}[Theoretical undersmoothing conditions]\label{lem:ucondition}
Let $\hat h_{a,j,\lambda_n^a}$ and $\hat h_{c,j,\lambda_n^c}$ denote highly adaptive lasso estimators of $h_{a,j}$ and $h_{c,j}$ with $L_1$-norm bounds $\lambda_n^a$ and $\lambda_n^c$ chosen such that, for $j=1,2,\cdots,l(x)$,
\begin{align}
    \min_{k \in \mathcal{J}_n^a} {\bigg \Vert} P_n \frac{d}{d\logit
    \hat h_{a,j,\lambda_n^a}}
    L(\logit \hat h_{a,j,\lambda_n^a}) (\phi_{k}) {\bigg \Vert}&= o_p\left(n^{-1/2}\right),  \label{eq:basisa} \\
    \min_{k \in \mathcal{J}_n^c} {\bigg \Vert} P_n \frac{d}{d\logit
    \hat h_{c,j,\lambda_n^c}}
    L(\logit \hat h_{c,j,\lambda_n^c}) (\phi_{k}) {\bigg \Vert} &= o_p\left(n^{-1/2}\right), \label{eq:basisc} 
\end{align}
where  $L(\cdot)$  is log-likelihood loss and $\phi_{k}$ denote the $k$th basis function (i.e., feature) included in the model. Also,   $\mathcal{J}_n^.$ is a set of indices
for the basis functions with nonzero coefficients in the corresponding model. Let  $D^a(f^a,\hat h_{a,j})
= f^a \cdot (A_j - \hat h_{a,j})$, and $D^c(f^c,\hat h_{c,j})
= f^c \cdot (C_j - \hat h_{c,j})$. The functions $f^a$ and $f^c$ are c\`{a}dl\`{a}g with finite sectional
variation norm. Let  $\tilde{f}_\phi^a$ and $\tilde{f}_\phi^c$ be projections of  $\tilde f^a$  and $ \tilde f^c$ onto the linear
span of the basis functions $\phi_{j}$ in $L^2(P)$, where $\phi_{s,j}$
satisfies condition (\ref{eq:basisa}) and (\ref{eq:basisc}), respectively. 
Then,  we have   $P_n D^a({f}^a,\hat h_{a,j}) = o_p\left(n^{-1/2}\right)$ and $P_n D^c({f}^c,\hat h_{c,j}) = o_p\left(n^{-1/2}\right)$.
\end{lemma}
\begin{proof}
    The proof is similar to the proof of Lemma 1 in \cite{ertefaie2023nonparametric}, and thus is omitted. 
\end{proof}

We now return to the proof of Theorem~\ref{ch3-theorem3}. We first show the consistency of the direct inverse probability weighting estimator with highly adaptive lasso. Let $\breve \Psi_{H_{a,c}}(x;\eta)=\frac{\Delta^{\eta}_{a}(x)\Delta_{c}(x)}{H_{a,c}(x)}\{I(U > x) - S_U(x; \eta)\}$ and  $\Psi_{H_a,H_c}(x;\eta)= \P_n \breve\Psi_{H_{a,c}}(x;\eta)$ where  $\P_n$ is the empirical average and $ H_{a,c,i}(x)={H}_{a, i}(x){H}_{c, i}(x)$.   For simplicity of notation, we will suppress $\eta$ in our notation such that $\breve\Psi_{H_{a,c}}(x;\eta) \equiv \breve\Psi_{H_{a,c}}(x)$.  We have
\begin{align} \label{eq:decompose hal1} \tag{A.13}
    \begin{split}
        \P_n \breve\Psi_{\hat H_{a,c}} - \mathbb{P} \breve\Psi_{H_{a,c}}
       = (\mathbb{P}_n - \mathbb{P}) \breve\Psi_{H_{a,c}} + \mathbb{P}(\breve\Psi_{\hat H_{a,c}} - \breve\Psi_{H_{a,c}}) + (\mathbb{P}_n - \mathbb{P})(\breve\Psi_{\hat H_{a,c}}  - \breve\Psi_{H_{a,c}}).
    \end{split}
\end{align}
Using empirical process theory and by consistency of the  highly adaptive lasso estimate $\breve\Psi_{\hat H_{a,c}}$, the third term will be $o_p(n^{-1/2})$ \citep{vaart1996weak, kosorok2008introduction}. The first term is already linear so it will contribute to the influence function. Hence, the asymptotic linearity fo our estimator relies on the linearity of the second term.  It can be written as 

\begin{align} \label{eq:asy hal1} \tag{A.14}
    \begin{split}
    \mathbb{P}_0(\breve\Psi_{\hat H_{a,c}} - \breve\Psi_{H_{a,c}}) 
    &= \mathbb{P}_0 \left\{ \frac{\Delta^\eta_{a}(x) \Delta_{c}(x) I(U > x)}{\hat{H}_{a,c}(x)} - \frac{\Delta^\eta_{a}(x) \Delta_{c}(x) I(U > x)}{H_{a,c}(x)} \right\} \\
    &= \mathbb{P}_0 \left( Q(x) \left[ - \frac{ \hat{H}_{a,c}(x) - H_{a,c}(x)}{H_{a,c}(x)} + \frac{ (\hat{H}_{a,c}(x) - H_{a,c}(x))^2}{\hat{H}_{a,c}(x)} \right] \right) \\
    &= \mathbb{P}_0 \sum_{j=0}^{l(x)} \Bigg[ \frac{Q_j(x) I(\bar{A}_{j-1} = \bar{g}_{j-1}^\eta) I(C_{j-1} = 0)}{ \prod_{s \le j} \left\{ 1 - h_{c,s}(1 \mid \bar{z}_s, \bar{a}_s) \right\} \prod_{s \le j} \left\{ h_{a,s}(g_j^\eta \mid \bar{z}_s, \bar{a}_s) \right\}}  \\
    & \hspace{2in} \quad  \left\{ I(A_j = g_j^\eta) I(C_j = 0) - \hat{h}_{a,c,j} \right\} \Bigg]  + o_p(n^{1/2}),
    \end{split}
\end{align}
where $Q_j(x)= E\{I(U  >x) | \bar{A}_{j}=\bar g_j^\eta,C_{j-1}=0,\bar{Z}_j \}$ and $h_{a,c,j}=h_{a,j}(g_j^\eta\mid\bar{z}) h_{c,j}(0\mid\bar{a}, \bar{z})$. 
Recall our definitions,
\begin{eqnarray*}
    h_{a,j}(g_j^\eta\mid\bar{z}_j) &=& (1-A_{j-1})P(A_j = g_j^\eta \mid A_{j-1} = 0, C_{j-1} = 0, Y_j = 1, \bar{Z}_j = \bar{z})+A_{j-1},\\
    h_{c,j}(c_j\mid\bar{a}_j, \bar{z}_j) &=& (1-C_{j-1})P(C_j = c_j \mid \bar{A}_{j} = \bar{a}_j, C_{j-1} = 0, Y_j = 1, \bar{Z}_j = \bar{z}_j)+C_{j-1}.
\end{eqnarray*}
The third equality in (\ref{eq:asy hal1}) follows from (i) the projection of the terms in the second equality onto the nuisance tangent space at each time point; and (ii) the rate of convergence of the highly adaptive lasso estimate (i.e., $\|\hat{ H}_{a,c}(x) - { H}_{a,c}(x) \|_2=o_p(n^{-1/4})$).   For simplicity of presentation, let 
\begin{eqnarray*}
\lefteqn{D_{CAR}(Q,h_{a,c}) = }\\
&&\sum_{j=0}^{l(x)} \left[ \frac{Q_j(x) I(\bar A_{j-1}=\bar g_{j-1}^\eta) I(C_{j-1}=0) }{\prod_{s \le j}\{1 - h_{c,s}(1\mid\bar{z}_s, \bar{a}_s)\}\prod_{s \le j}\{ h_{a,s}(g_j^\eta\mid\bar{z}_s, \bar{a}_s)\}} \left\{ I(A_j=g_j^\eta) I(C_j=0) - \hat h_{a,c,j}\right\}\right] 
\end{eqnarray*}
We first rewrite
\begin{eqnarray*}
 \lefteqn{I(A_j=g_j^\eta) I(C_j=0) - \hat h_{a,c,j} = }\\
 &&I(A_j=g_j^\eta)\{  I(C_j=0) - \hat h_{c,j}(0\mid\bar{a}, \bar{z})\} +\hat h_{c,j}(0\mid\bar{a}, \bar{z})\{ I(A_j=g_j^\eta)  - \hat h_{a,j}(g_j^\eta\mid\bar{z})\}. 
\end{eqnarray*}
Then define
\begin{align*}
    f^a_j&= \frac{Q_j(x) I(\bar A_{j-1}=\bar g_{j-1}^\eta) I(C_{j-1}=0) \hat h_{c,j}(0\mid\bar{z}_s, \bar{a}_s) }{\Pi_{s \le j}\{1 - h_{c,s}(1\mid\bar{z}_s, \bar{a}_s)\}\Pi_{s \le j}\{ h_{a,s}(g_j^\eta\mid\bar{z}_s, \bar{a}_s)\}}\\
    f^c_j&= \frac{Q_j(x) I(\bar A_{j-1}=\bar g_{j-1}^\eta) I(C_{j-1}=0) I(A_j=g_j^\eta) }{\Pi_{s \le j}\{1 - h_{c,s}(1\mid\bar{z}_s, \bar{a}_s)\}\Pi_{s \le j}\{ h_{a,s}(g_j^\eta\mid\bar{z}_s, \bar{a}_s)\}}.
\end{align*}
Accordingly, 
\begin{align*}
    \mathbb{P}_0(\breve\Psi_{\hat H_{a,c}} - \breve\Psi_{H_{a,c}}) = -(\P_n-\P_0) D_{CAR}(Q,h_{a,c})&+\P_n \sum_{j=0}^{l(x)} f^a_j \left\{I(A_j=g_j^\eta)  - \hat h_{a,j}(g_j^\eta\mid\bar{z})\right\} \\
        &+\P_n \sum_{j=0}^{l(x)} f^c_j \left\{I(C_j=0)  - \hat h_{c,j}(0\mid \bar{a}_j,\bar{z}_j)\right\}+o_p(n^{1/2}).
\end{align*}

 Gathering all the terms together, we have
\begin{align} \label{eq:decompose hal1} \tag{A.15}
    \begin{split}
        \P_n \breve\Psi_{\hat H_{a,c}} - \mathbb{P} \breve\Psi_{H_{a,c}}
       = (\mathbb{P}_n - \mathbb{P}) \breve\Psi_{H_{a,c}}  -&(\P_n-\P_0) D_{CAR}(Q,h_{a,c})+\P_n \sum_{j=0}^{l(x)} f^a_j \left\{I(A_j=g_j^\eta)  - \hat h_{a,j}(g_j^\eta\mid\bar{z})\right\} \\
        &+\P_n \sum_{j=0}^{l(x)} f^c_j \left\{I(C_j=0)  - \hat h_{c,j}(0\mid \bar{a}_j,\bar{z}_j)\right\}+o_p(n^{1/2}).
    \end{split}
\end{align}
Hence, if we undersmooth $\hat h_{c,j}(1\mid \bar{a}_j,\bar{z}_j)$ and $\hat h_{a,j}(1\mid\bar{z}_j)$, for each $j=1,2,\cdots,l(x)$, such that $\left| \P_n \sum_{j=0}^{l(x)} f^a_j \left\{I(A_j=g_j^\eta)  - \hat h_{a,j}(g_j^\eta\mid\bar{z})\right\} \right|=o_p(n^{-1/2})$ and 
$\left| \P_n \sum_{j=0}^{l(x)} f^c_j \left\{I(C_j=0)  - \hat h_{c,j}(0\mid \bar{a}_j,\bar{z}_j)\right\} \right|=o_p(n^{-1/2})$, the estimator $\P_n \breve\Psi_{\hat H_{a,c}}$ will be asymptotically linear. Subsequently, we have 
\begin{align}    
    \hat{S}_U(x;\eta) - S_U(x;\eta) = & (\mathbb{P}_n - \mathbb{P}) \breve\Psi_{H_{a,c}}  -(\P_n-P_0) D_{CAR}(Q,h_{a,c}) + o_p(n^{1/2}).
\end{align}
Denote $IC_{HAL}(O;x,\eta) = (\mathbb{P}_n - \mathbb{P}) \breve\Psi_{H_{a,c}}  -(\P_n-P_0) \left[ Q(x) \frac{\{\Delta^\eta_{a}(x) \Delta_{c}(x) - { H}_{a,c}(x)\}}{{ H}_{a,c}(x)}\right]$. 
Then, the RQAL with highly adaptive lasso will be asymptotic linear with
\begin{align} \tag{A.15}
\begin{split}
     n^{1/2}\left\{\hat{R}(\hat{\eta}_{opt}) - R(\eta_{opt})\right\} =& n^{-1/2} \sum_{i=1}^{n} \int_{0}^{L_U} IC_{HAL}(o_i; x, \eta_{opt}) dx + o_p(1).
\end{split}
\end{align}

Similar arguments in the proof of Theorem~\ref{ch3-theorem1} can be used to show the asymptotic equivalence between $\hat{S}_U(x; \hat{\eta}_{opt})$ and $\tilde{S}_U(x; \tilde{\eta}_{opt})$.

\end{proof}

\subsection{The Highly Adaptive Lasso} \label{app:hal}

The basic idea is to model an outcome on a $d_j$-dimensional covariate vector $X$ by mapping $X$ into a group of indicator basis functions.  Let $s \subset (1, ..., d)$, $\mathcal{T} = \{s: s \subset (1, ..., d)\}$, and denote $X_i^*(s) = \{X_{il}: l\in s \}, i = 1,...,n$. For a typical observation $x$, we generate the basis expansion as $\Phi_n(x) = \{\phi_1(x), ..., \phi_n(x)\}^{\top}$, where $\phi_i(x) = \{\phi_{s,i}(x): s \in \mathcal{T}\}$ and $\phi_{s,i}(x) = \{x \rightarrow I(x^*(s) \ge X_i^*(s))\}$. For example when $X = (X_1, X_2)$, $\phi_i(x) = \{\phi_{1,i}(x), \phi_{2,i}(x), \phi_{12,i}(x)\} = \{I(x_1 \ge X_{i1}), I(x_2 \ge X_{i2}), I(x_1 \ge X_{i1}, x_2 \ge X_{i2})\}, i = 1,...,n,$ is the collection of the first- and the second-order indicator basis functions.

We now discuss the highly adaptive lasso fit of the treatment assignment model.  Let $h_{a,j} \in \mathcal{D}_j[0,\tau]$, $j = 1, ..., K$, where $\mathcal{D}_j[0,\tau]$ be the Banach space of $d_j$-variate real-valued c\`{a}dl\`{a}g
 (right-continuous with left-hand limits) functions on a cube $[0,\tau] \in \R^{d_j}$, where $\tau$ is the upper bound of all supports and is assumed to be finite. For a given subset  $s \subset \{0,1,\ldots,d_j\}$, define $u_s =(u_l :l \in s)$ and $u_{-s}$ as the complement of $u_s$. Then, $h_{a,j,s}: [0_s,\tau_s] \rightarrow \R$ defined as $h_{a,j,s}(u_s) = h_{a,j,s}(u_s,0_{-s})$. Thus, $h_{a,j,s}(u_s)$  is a section of function $h_{a,j}$ that sets the components in the complement of $s$ equal to zero and varies only along the variables in $u_s$.

 Following \cite{gill1995inefficient}, $\logit h_{a,j}$ can be represented as
 \begin{align}
\text{logit } h_{a,j}(\bar z_j):&=\text{logit } h_{a,j}(0)+\sum_{s \subset\{1,\ldots,d_j\}} \int_{0_s}^{\bar{z}_{j,s}} d\logit h_{a,j,s}(u_{s}) \nonumber \\
      &=\text{logit } h_{a,j}(0)+\sum_{s \subset\{1,\ldots,d_j\}} \int_{0_s}^{\tau_s} I(u_{s} \leq \bar z_{j,s})d\logit h_{a,j,s}(u_{s}). \label{eq:hal1}
\end{align}
 Let   $\phi_{s,i}(c_s)= I(\bar z_{j,s,i} \leq c_s)$ where $\bar z_{j,s,i} $ are support points of $dh_{a,j,s}$. The representation (\ref{eq:hal1}) can be approximated
$
 \text{logit } h_{a,j,\beta} = \beta^a_{j,0}+\sum_{s \subset\{1,\ldots,d_j\}}\sum_{i=1}^{n} \beta^a_{j,s,i} \phi_{s,i},
$ 
where $ |\beta^a_{j,0}|+\sum_{s \subset\{1,\ldots,d_j\}}\sum_{i=1}^{n} |\beta^a_{j,s,i}|$ is an approximation of the sectional variation norm of $\logit h_{a,j}$. The highly adaptive lasso defines a minimum loss based estimator $\hat \beta^a_j$  as
\[
\hat \beta^a_j= \arg \min_{\beta: |\beta^a_{j,0}|+\sum_{s \in \mathcal{T}}\sum_{i=1}^{n} |\beta^a_{j,s,i}|<\lambda} \mathbb{P}_n L(\text{logit } h_{a,j,\beta}),
\]
where $L(\cdot)$ is the negative log-likelihood loss function and $\mathbb{P}_n$ is the empirical average. We denote the estimated treatment assignment model at time $j$ as $\hat h_{a,j,\hat \beta^a}\equiv \hat h_{a,j}$.
The sectional variation norm $\lambda$ (i.e., $l_1$-norm of the coefficient vector) is unknown and must be estimated using the observed data. The finite sectional variation norm assumption implies that there is a positive constant $ 0<M<\infty$ such that $0<\lambda < M$. In practice, $\lambda$ is determined using cross-validation which is  justified theoretically by the oracle properties of cross-validation \citep{van2003unified, van2006oracle}. 
Similarly, we define a minimum loss based estimator $\hat \beta^c_j$  as
\[
\hat \beta^c_j= \arg \min_{\beta: |\beta^c_{j,0}|+\sum_{s \in \mathcal{T}}\sum_{i=1}^{n} |\beta^c_{j,s,i}|<\lambda} \mathbb{P}_n L(\text{logit } h_{c,j,\beta}).
\]
We denote the estimated discrete hazard function for censoring at time $j$ as $\hat h_{c,j,x,\hat \beta^c}\equiv \hat h_{c,j,x}$.

\subsection{Plots for Simulation Section} \label{app:addplot}
The first row of these figures shows the distribution of true RQAL of each estimated optimal regime compared to the true RQAL of the true optimal regime (the dashed line). The second row shows the estimated RQAL of the estimated optimal regime. The third row compares the treatment regime recommendation between the true optimal regime and the estimated optimal regime by computing the mis-classification rate (MR).

\begin{figure}[H]
\centering
\captionsetup{font=small}
  \includegraphics[width=1\textwidth]{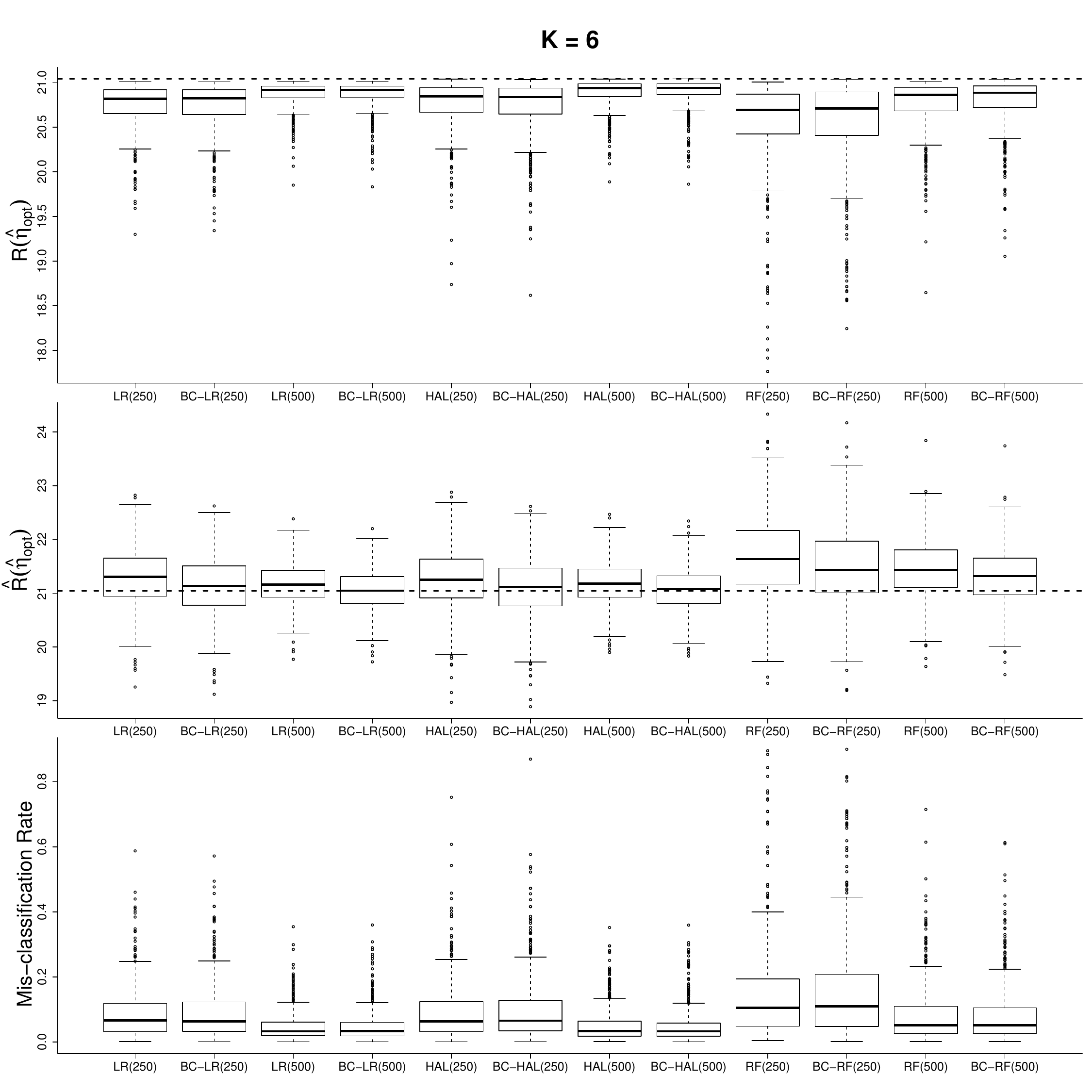}
   \caption{The simulation results of the estimated RQAL in scenario 1 with K =6. IPW is the estimator \eqref{eq:rqaldirect} while BC-IPW is the bias-corrected estimator \eqref{eq:rqalsmooth}. The number in brackets is the sample size. }
  \label{fg:s1}
\end{figure}

\begin{figure}[H]
\centering
\captionsetup{font=small}
  \includegraphics[width=1\textwidth]{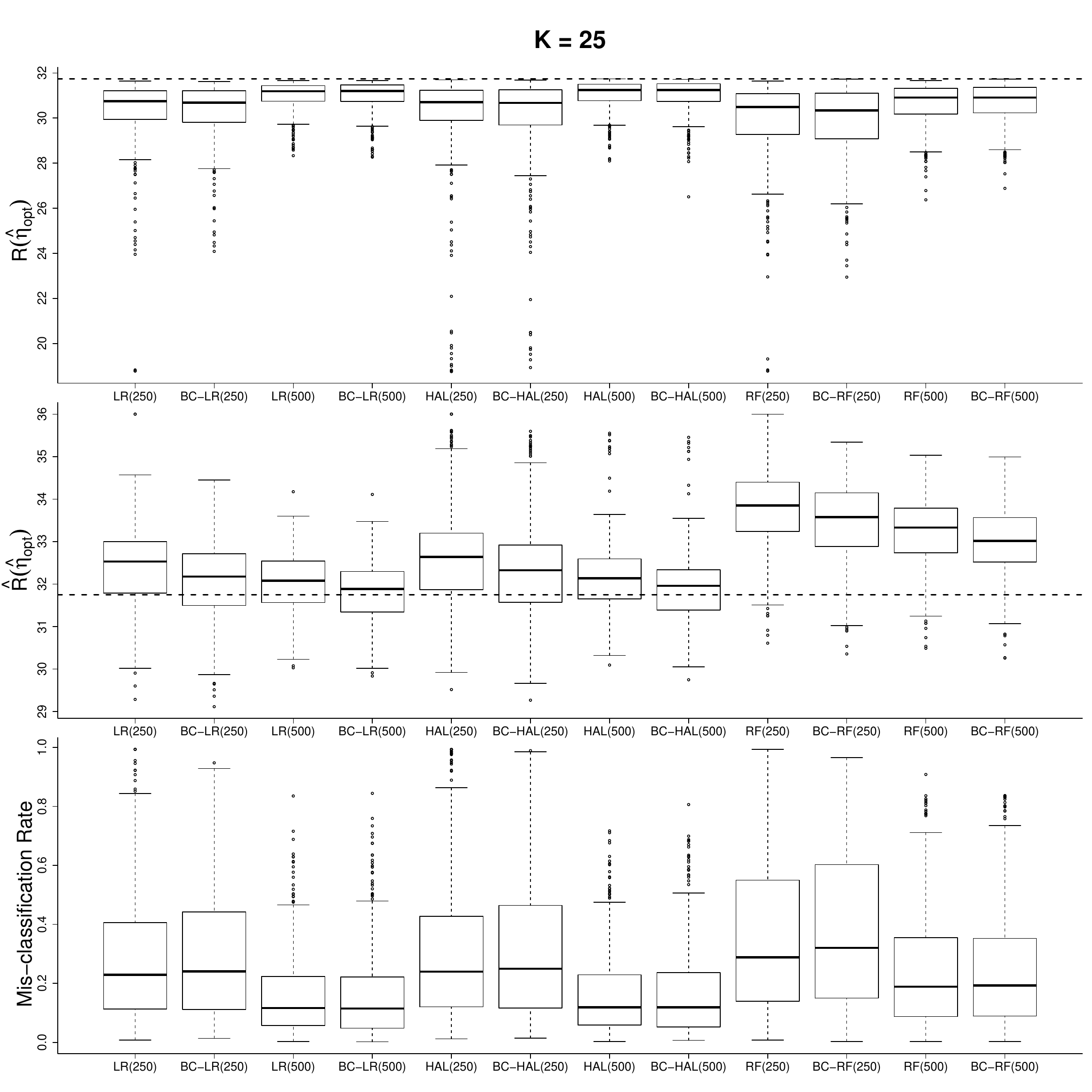}
   \caption{The simulation results of the estimated RQAL in scenario 1 with K = 25. IPW is the estimator \eqref{eq:rqaldirect} while BC-IPW is the bias-corrected estimator \eqref{eq:rqalsmooth}. The number in brackets is the sample size. }
  \label{fg:s2}
\end{figure}

\begin{figure}[H]
\centering
\captionsetup{font=small}
  \includegraphics[width=1\textwidth]{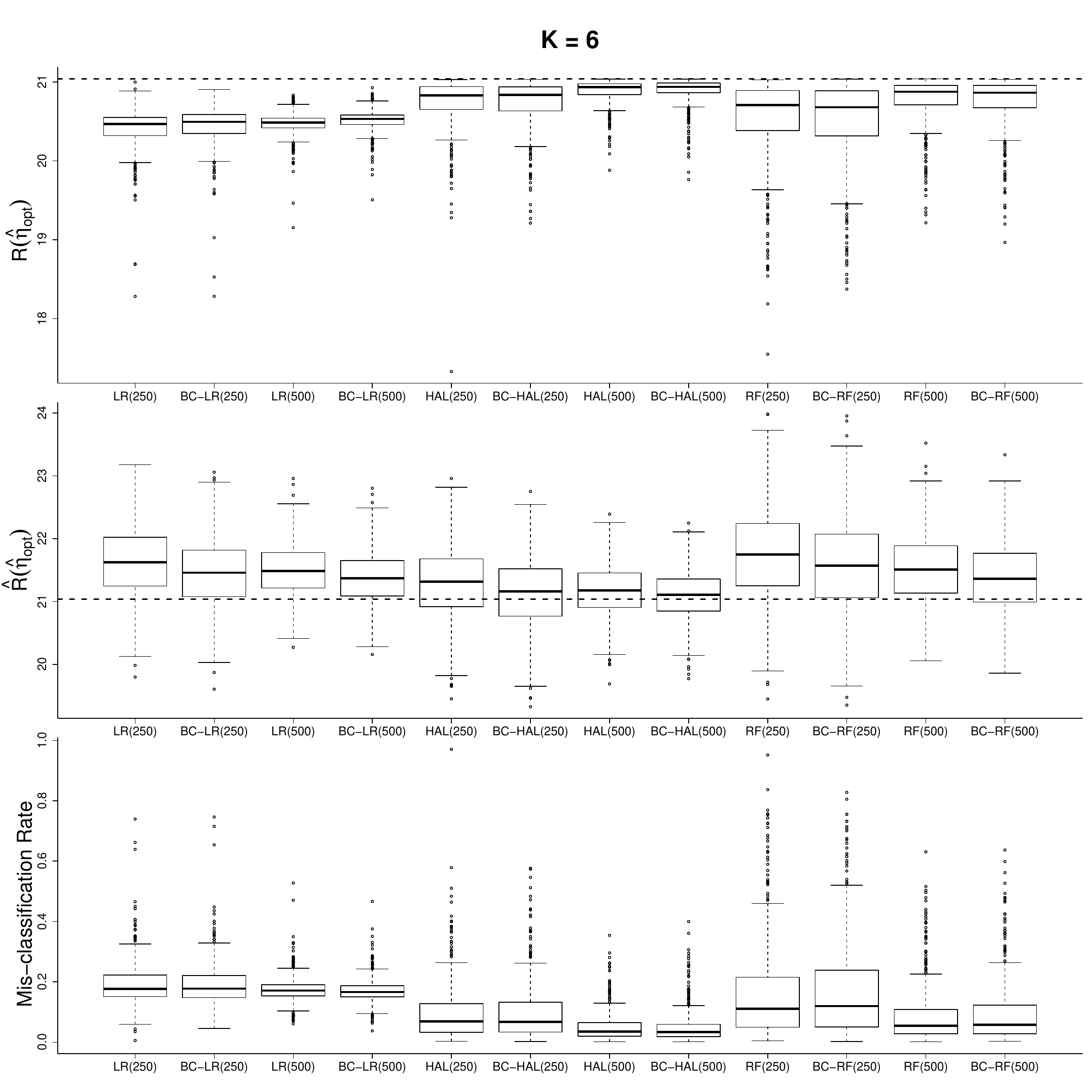}
   \caption{The simulation results of the estimated RQAL in scenario 2 with K = 6. IPW is the estimator \eqref{eq:rqaldirect} while BC-IPW is the bias-corrected estimator \eqref{eq:rqalsmooth}. The number in brackets is the sample size. }
  \label{fg:s3}
\end{figure}

\begin{figure}[H]
\centering
\captionsetup{font=small}
  \includegraphics[width=1\textwidth]{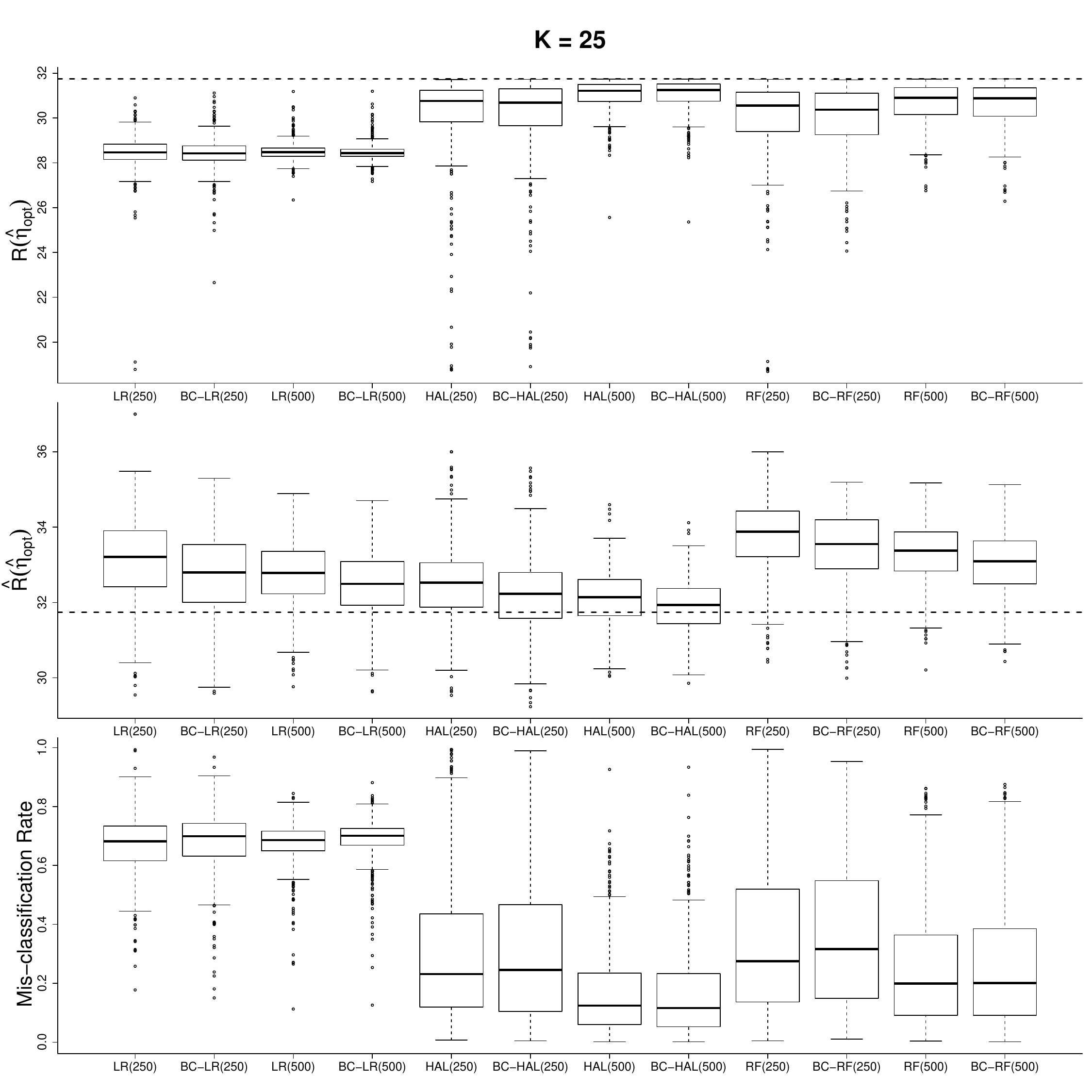}
   \caption{The simulation results of the estimated RQAL in scenario 2 with K =25. IPW is the estimator \eqref{eq:rqaldirect} while BC-IPW is the bias-corrected estimator \eqref{eq:rqalsmooth}. The number in brackets is the sample size. }
  \label{fg:s4}
\end{figure}

\subsection{ Simulation Results for Scenario 1} \label{app-scenario1}

  In the first set of simulation, we set $L = 60$ and $G = 10$. The mean observed survival time is around 32.4 weeks; the mean quality-adjusted lifetime is around 18.4 weeks. In this setting, $\sim$66\% patients initialize the treatment in the follow-up, and $\sim$29\% of patients  follow the optimal regime $g^{\eta_{opt}}$ exactly. We choose $L_U = 26$ weeks as the upper bound of restricted quality-adjusted lifetime for identifiability which lies in the 90\% quantile of the observed quality-adjusted lifetime and the theoretical $R(\eta_{opt})$ = 21.04 weeks in this case.  
  
 In the second set of simulation, we set $L = 100$ and $G = 4$. The mean observed survival time is around 47 weeks; the mean quality-adjusted lifetime is around 25.2 weeks. In this setting,  $\sim$55\% patients initialize the treatment in the follow-up, and $\sim$20\% of patients follow the optimal regime $g^{\eta_{opt}}$ exactly. We choose $L_U = 36$ weeks as the upper bound of restricted quality-adjusted lifetime which lies in the 90\% quantile of the observed quality-adjusted lifetime and the true $R(\eta_{opt})$ = 31.74 weeks in this case. 

 \begin{remark}
    To simplify the presentation of different estimators in the tables, we use the notation $\breve R$ and $\breve \eta$  to represent the estimated quantities corresponding to the methods specified in each row. For example, in Table \ref{App:tb:4.1},  $\breve R(\breve \eta_{opt})$ in row 2 denotes $\hat R^{npar}(\hat \eta^{npar}_{opt})$, whereas in row 4, it denotes $\tilde R^{par}(\tilde \eta^{par}_{opt})$.
\end{remark}

 {\bf Table \ref{App:tb:4.1} Specifications:} The columns $(\eta_0, \eta_1, \eta_2)$ correspond to the average of estimated parameters used to form the optimal strategy and the numbers in parenthesis are the empirical standard errors. IPW is the direct estimation using \eqref{eq:rqaldirect}, BC-IPW is the bias-correction estimator \eqref{eq:rqalsmooth}; Logit, logistic regression, HAL, undersmoothed highly adaptive lasso, RF, random forest; $K$ is the number of stages; $n$ is the sample size; $\eta_0$, $\eta_1$, $\eta_2$ are the empirical average of $\breve{\eta}_0/\breve{\eta}_1$, $\breve{\eta}_1/\breve{\eta}_2$, $\breve{\eta}_2/\breve{\eta}_0$; $\breve{R}(\breve{\eta}_{opt})$ is the empirical average of  estimated RQAL of the estimated optimal regime. MR is the empirical misclassification rate of the estimated optimal regime.

 The misclassification rates of both highly adaptive lasso and logistic regression are also very close. Among the methods considered, the random forest demonstrates the poorest performance in both bias and empirical standard errors. Notably, the conservative standard error presented at the end of Section \ref{sec:nonparm} underestimates the empirical standard errors of $\hat{R}(\hat{\eta}_{opt})$, as indicated in parentheses, resulting in significant undercoverage (Figure \ref{fg:4.1}). The misclassification rates are also substantially higher than the other estimators. 

 Figures \ref{fg:s1} and \ref{fg:s2} in Section \ref{app:addplot} provide additional information about the estimators for $K=6$ and $K=25$, respectively. Specifically, it provides insight about the distribution of (a) the true restricted quality-adjusted lifetime under different estimated optimal regimes (i.e., $R(\hat{\eta}_{opt})$ and $R(\tilde{\eta}_{opt})$); (b)  the estimated restricted quality-adjusted lifetime under the estimated optimal regime (i.e., $\hat R(\hat{\eta}_{opt})$ and $\tilde R(\tilde{\eta}_{opt})$); and, (c) the misclassification rates. Overall the plots show the superiority of the correctly specified logistic regression and the highly adaptive lasso compared with random forest. Also, there is no evidence of deviating from normal distribution for $\hat{R}(\hat{\eta}_{opt})$, $\tilde{R}(\tilde{\eta}_{opt})$, ${R}(\hat{\eta}_{opt})$ and ${R}(\tilde{\eta}_{opt})$ which confirms our theoretical results in Theorems \ref{ch3-theorem2} and \ref{ch3-theorem3}.

\begin{table}[htbp!]
\centering
\captionsetup{font=small}
\caption{Summary table of simulation studies in Scenario 1.  The target parameter here is $R(\eta_{opt}) = 21.04$ when $K = 6$ and $R(\eta_{opt}) = 31.74$ when $K = 25$; SE is the standard error estimation; CP is the average coverage probability; $R(\hat{\eta}_{opt})$ is the average of true RQAL of the estimated optimal regime. The numbers in brackets are the empirical standard deviation. } \label{App:tb:4.1}
\scalebox{0.8}{
\begin{tabular}{ll|l|l|llllllll}
\hline
                       &       & K                & n                    & $\eta_0$ & $\eta_1$ & $\eta_2$ & $\hat{R}(\hat{\eta}_{opt})$ & SE    & CP    & $R(\hat{\eta}_{opt})$ & MR         \\ \hline
\multirow{3}{*}{IPW}   & Logit & \multirow{12}{*}{6}  & \multirow{6}{*}{250} & 1.03(0.16)&1.03(0.28)&1.00(0.13)   & 21.30(0.55)      & 0.55   & 0.90  & 20.74(0.25)           & 9.12(8.26)    \\
                       & HAL   &                      &                      & {1.02(0.16)}&{1.05(0.32)}&{0.99(0.14)}   & {21.29(0.58)} & {0.55}   &{0.90}  &  {20.70(0.29)}                    &  {10.42(9.35)}  \\
                       & RF    &                      &                      & 1.04(0.30)&1.03(0.51)&1.01(0.24)   & 21.67(0.75)      & 0.67    & 0.80   & 20.55(0.49)           & 14.88(15.15) \\ \cline{5-12} 
\multirow{3}{*}{BC-IPW} & Logit &                     &                      & 1.02(0.17)&1.06(0.36)&0.99(0.15)   & 21.14(0.55)      & 0.55   &  0.94 & 20.73(0.27)            & 9.31(8.68)   \\
                       & HAL   &                      &                      & {1.02(0.19)}&{1.06(0.36)}&{1.00(0.16)}   & {20.93(0.52)}      & {0.56}  & {0.97}  & {20.70(0.29)}                    &  {10.46(9.54)}  \\
                       & RF    &                      &                      & 1.03(0.37)&1.09(0.67)&1.03(0.39)   & 21.49(0.75)      & 0.75  & 0.84   & 20.55(0.49)                  & 15.71(15.45) \\ \cline{1-2} \cline{4-12} 
\multirow{3}{*}{IPW}   & Logit &                      & \multirow{6}{*}{500} & 1.00(0.08)&1.02(0.15)&0.99(0.07)   & 21.17(0.39)      & 0.39 & 0.92 & 20.87(0.14)                    &  4.87(4.69)  \\
                       & HAL   &                      &                      &{0.99(0.07)}&{1.03(0.16)}&{0.99(0.08)}   & {21.17(0.41)}     & {0.39} & {0.92} &      {20.85(0.16)}                & {5.39(5.15)}    \\
                       & RF    &                      &                      & 1.01(0.15)&1.05(0.30)&0.99(0.13)   & 21.44(0.54)      & 0.50  & 0.86      & 20.76(0.28)          & 8.42(9.15)    \\ \cline{5-12} 
\multirow{3}{*}{BC-IPW} & Logit &                     &                      & 1.01(0.08)&1.02(0.17)&1.00(0.08)   & 21.06(0.39)      & 0.39 & 0.95 & 20.86(0.15)                     & 5.02(5.01)    \\
                       & HAL   &                      &                      & {1.01(0.09)}&{1.01(0.16)}&{1.00(0.08)}   & {20.92(0.37)}      & {0.40} & {0.95} &  {20.84(0.15)}                    &   {5.76(4.86)} \\
                       & RF    &                      &                      & 1.01(0.17)&1.07(0.36)&0.99(0.15)   & 21.31(0.54)      & 0.51  & 0.89   & 20.78(0.28)                   &   8.48(9.15)        \\ \hline
\multirow{3}{*}{IPW}   & Logit & \multirow{12}{*}{25} & \multirow{6}{*}{250} & 1.05(0.29)&1.06(0.45)&1.00(0.19)   & 32.40(0.90)      & 0.86 & 0.82 & 30.33(1.47)                    &  28.66(21.65)\\
                       & HAL   &                      &                      & {1.05(0.37)}&{1.10(0.65)}&{0.99(0.27)}   & {32.37(0.83)}      & {0.88} & {0.85} &              {30.14(1.39)}      & {33.09(24.18)} \\
                       & RF    &                      &                      & 1.10(0.53)&1.18(1.34)&1.02(0.30)   & 33.74(0.84)      & 0.72  & 0.28      & 29.95(1.74)          & 35.66(25.54)  \\ \cline{5-12} 
\multirow{3}{*}{BC-IPW} & Logit &                     &                      & 1.06(0.34)&1.08(0.55)&1.00(0.22)   & 32.11(0.88)      & 0.88 & 0.87 &  30.34(1.20)                 & 29.62(22.35) \\
                       & HAL   &                      &                      & {1.07(0.44)}&{1.06(0.58)}&{1.01(0.27)}   & {31.59(0.86)}      & {0.92}& {0.95} &  {30.26(1.06)}               & {32.53(22.61)} \\
                       & RF    &                      &                      & 1.01(0.84)&1.02(1.76)&1.01(0.39)   & 33.46(0.90)      & 0.77  & 0.39   &  29.94(1.51)                    &  37.70(26.13) \\ \cline{1-2} \cline{4-12} 
\multirow{3}{*}{IPW}   & Logit &                      & \multirow{6}{*}{500} & 1.01(0.10)&1.01(0.19)&1.00(0.09)   & 32.05(0.68)      & 0.71 & 0.89 & 31.01(0.60)                 &  15.92(13.85)\\
                       & HAL   &                      &                      & {1.02(0.14)}&{1.02(0.23)}&{1.00(0.12)}   & {32.07(0.71)}      & {0.69} &      {0.86}             & {30.89(0.75)}                 & {18.44(16.90)} \\
                       & RF    &                      &                      & 1.06(0.30)&1.00(0.39)&1.03(0.19)   & 33.25(0.77)      & 0.74  & 0.44       & 30.62(0.92)          & 24.63(20.23)  \\ \cline{5-12} 
\multirow{3}{*}{BC-IPW} & Logit &                     &                      & 1.01(0.12)&1.02(0.23)&1.00(0.11)   &  31.83(0.68)             & 0.72 & 0.94 & 30.99(0.65)                    & 16.37(15.01) \\
                       & HAL   &                      &                      & {1.01(0.11)}&{1.02(0.21)}&{1.00(0.10)}   & {31.47(0.71)} & {0.73} & {0.94} &  {30.93(0.66)}       & {17.64(15.11)} \\
                       & RF    &                      &                      & 1.06(0.30)&1.03(0.44)&1.02(0.21)   & 32.98(0.79)     & 0.77  & 0.57   &      30.65(0.93)               & 25.02(20.67) \\ \hline
\end{tabular}}
\end{table}

\subsection{Additional Simulation Results for Undersmoothed HAL} \label{app-sim}
In this section, we show the necessity of using undersmoothed HAL when modeling nuisance parameters by comparing with cross-validation (CV)-HAL. 
\subsubsection{One Stage Study}
In this example, we consider a one stage study with Gaussian outcome, and the purpose is to estimate $Y(1)$ which is 0.
\begin{itemize}
    \item \emph{Outcome model:} $Y = \alpha_1Z_1 + \alpha_2Z_2 + \epsilon$.
    \item \emph{Treatment assignment model:} $\mathbb{P}(A = 1 \mid Z_1, Z_2) = \exp(-\beta_1 Z_1 - \beta_2 Z_2)^{-1}$.
\end{itemize}
We generate the data as follows:
\begin{itemize}
    \item[](1) Generate $Z_1$, $Z_2 \sim N(0, 0.5)$, and the error term $\epsilon \sim N(0, 0.5)$.
    \item[](2) Generate the treatment variable $A$ by the treatment assignment model.
    \item[](3) Generate the outcome variable $Y$ by the outcome model.
    \item[](4) Set the observed covariates as $X_1 = \exp(Z_1/2)$ and $X_2 = (Z_1 + Z_2)^2$.
    \item[](5) The observed data is $(X_1, X_2, Y, A)$.
\end{itemize}
Here, $\alpha = (\alpha_1, \alpha_2) = (1, 1)$ and $\beta = (\beta_1, \beta_2) = (-1.5, 1.5)$. The PS model would be mis-specified if we model treatment assignment by logistic regression and the observed covariates. In Table~\ref{tb:appendix-sim} , we show the results of IPW estimators by using CV-HAL and undersmooth HAL for the PS model with sample size $n = 500$.
The standard error estimation for logistic regression is calculated by sandwich formula.

\subsubsection{Multi Stage Study}
In this example, we consider a $K = 6$ stage study which is similar to the example in Section~\ref{ch3-sec:simu} and the outcome of interest is RQAL $R(\eta_{opt})$. The optimal treatment regime is $\eta_{opt} = (1, -1, -1)$. We ignore the censoring part and the estimation of $\eta_{opt}$ for simplicity, and assume the optimal treatment regime is known. The purpose is to evaluate the performance of the IPW estimator \eqref{eq:rqaldirect} under CV-HAL and undersmooth HAL for the PS model. We still consider the case when PS model is mis-specified by logistic regression with observed covariates. The covariates $Z_j = (Z_{j1}, Z_{j2}, Z_{j3})$ are generated as described in Section~\ref{ch3-sec:simu}.
\begin{itemize}
    \item \textit{Treatment assignment model}: $h_{a, j}(\bar{z}) = logit^{-1}(\kappa_0 + \kappa_1 \times  z_{1j} + \kappa_2 \times z_{2j})$
\end{itemize}
where $(\kappa_0, \kappa_1, \kappa_2) = (- 0.1 \times K, -1, -1)$ and the observed covariates are $X_j = \{X_{j1}, X_{j2}, X_{j3}\} = \{\exp(Z_{j1}/2), (Z_{j1} + Z_{j2})^2, Z_{j3}\}$. We generate the survival time with the following discrete hazard model
\begin{align*}
    h_{y}(t, \bar{a}, \bar{z}) =&  pr(T = t | T > t - 1, \bar{A}_j = \bar{a}_j, \bar{C}_j = 0, \bar{Z}_j = \bar{z})\\
    =&  \text{logit}^{-1}\{-5 + t \times (6 \times 0.75)/K -0.5 \times z_{1j} -0.5 \times z_{2j} -0.5 \times a_j \times W_j\}\\
	& \text{for} \quad t \in (jG, (j+1)G)], \quad j = 0,..., K-1. 
\end{align*}
We then the define the utility map as $\mathcal{Q}(Z_{3j}) = Z_{3j}/\max_{k = 0, .., K} Z_{3k}$ which is always a number between (0, 1]. Patients' quality of life scores will reduce to zero if they are no longer at-risk. The upper limit for RQAL is selected as $L_U = 25$ by picking the 90\% quantile of observed QAL and $R(\eta_{opt}) = 20.73$ in this case. We use $n = 500$ in this example. See Table~\ref{tb:appendix-sim} of the results.

\begin{table}[htb]
\centering
\captionsetup{font=small}
\caption{Summary table to compare CV-HAL and undersmoothed HAL. The results are averaged over 500 Monte Carlo data sets. SD is the empirical standard error, SE is the averaged standard error estimation, CP is the averaged coverage probability, and $\lambda$ is the averaged tuning parameter in $l_1$ regression.} \label{tb:appendix-sim}
\scalebox{1}{
\begin{tabular}{lcccccc}
\hline
Example                      & Method    & Bias   & SD    & SE    & CP (\%) & $\lambda$\\ \hline
\multirow{3}{*}{One-stage}   & Logit     & 0.171  & 0.065 & 0.066 & 11       & NA\\
                             & CV-HAL    & 0.085  & 0.048 & 0.055 & 69       & 0.025\\
                             & Under-HAL & 0.006  & 0.054 & 0.054 & 94       & 0.001\\ \hline
\multirow{3}{*}{Multi-stage} & Logit     & 0.388  & 0.365 & 0.361 & 84       & NA\\
                             & CV-HAL    & 0.245  & 0.365 & 0.350 & 85       & 0.024\\
                             & Under-HAL & 0.021  & 0.387 & 0.364 & 94       & 0.007 \\ \hline
\end{tabular}}
\end{table}

\section{Additional Tables for the real data example} \label{app:addtable}

Table~\ref{app:tb:logit} shows the estimated coefficients of the measured variables on the probability of PEG tube insertion and censoring on logit scale.  Based on these results, the ALS bulbar and FVC are the most significant variables in deciding for PEG placement such that patients with higher values of bulbar and FVC are less likely to receive PEG. Also, the stage (i.e., number of months since the beginning of the follow-up), site of onset, Ceftriaxone, and ALS total score contribute significantly in hazard of censoring.

\begin{table}[t]
\centering
\captionsetup{font=small}
\caption{Summary table of logistic regression for PEG insertion and censoring. Stage represents the number of months since the beginning of the follow-up. Stage$^2$ and Stage$^3$ and the  square and cubic of Stage.} \label{app:tb:logit}
\scalebox{0.9}{
\begin{tabular}{lllllll}
\hline
                  & \multicolumn{3}{l}{PEG insertion}                    & \multicolumn{3}{l}{Censoring}     \\ \hline
                  & Est.   & S.E.   & \multicolumn{1}{l|}{p-val}            & Est.   & S.E.  & p-val            \\
(Intercept)       & 1.093  & 1.010  & \multicolumn{1}{l|}{0.279}            & -5.827 & 0.907 & \textless{}0.001 \\
Stage             & -0.054 & 0.098  & \multicolumn{1}{l|}{0.584}            & 0.849  & 0.131 & \textless{}0.001 \\
Stage$^2$            & -0.002 & 0.009  & \multicolumn{1}{l|}{0.781}            & -0.048 & 0.009 & \textless{}0.001 \\
Stage$^3$            & 0.000  & 0.000  & \multicolumn{1}{l|}{0.584}            & 0.001  & 0.000 & \textless{}0.001 \\
Sex               & -0.099 & 0.187  & \multicolumn{1}{l|}{0.595}            & 0.048  & 0.138 & 0.727         \\
Age               & 0.013  & 0.009  & \multicolumn{1}{l|}{0.131}            & -0.000 & 0.006 & 0.984            \\
Race              & -0.396 & 0.404  & \multicolumn{1}{l|}{0.327}            & 0.097  & 0.265 & 0.715            \\
Siteofonset       & 0.480  & 0.244  & \multicolumn{1}{l|}{0.049}            & 0.576  & 0.198 & 0.004            \\
Duration          & -0.007 & 0.014  & \multicolumn{1}{l|}{0.584}            & 0.012  & 0.009 & 0.189            \\
Ceftriaxone       & -0.340 & 0.172  & \multicolumn{1}{l|}{0.048}            & -0.292 & 0.124 & 0.019            \\
Riluzole          & 0.327  & 0.201  & \multicolumn{1}{l|}{0.105}            & -0.266 & 0.136 & 0.051            \\
BMI               & 0.002  & 0.016 & \multicolumn{1}{l|}{0.888}            & -0.011 & 0.012 & 0.361            \\
ALS total        & -0.03  & 0.042  & \multicolumn{1}{l|}{0.446}            & -0.074 & 0.029 & 0.010            \\
ALS bulbar       & -0.291 & 0.061  & \multicolumn{1}{l|}{\textless{}0.001} & 0.065  & 0.041 & 0.118            \\
ALS breathing    & -0.029 & 0.059  & \multicolumn{1}{l|}{0.627}            & 0.042  & 0.041 & 0.303            \\
ALS gross motor & -0.040 & 0.070  & \multicolumn{1}{l|}{0.56}             & -0.023 & 0.046 & 0.627            \\
FVC               & -0.016 & 0.004  & \multicolumn{1}{l|}{\textless 0.001}  & 0.004  & 0.003 & 0.156            \\
BMI change       & -0.012 & 0.010  & \multicolumn{1}{l|}{0.221}            & -0.000 & 0.005 & 0.967            \\
Grip              & -0.004 & 0.006  & \multicolumn{1}{l|}{0.460}            & 0.003  & 0.004 & 0.519           \\
ALSQOL            & 0.049  & 0.067  & \multicolumn{1}{l|}{0.465}            & -0.03  & 0.046 & 0.485            \\ \hline
\end{tabular}}
\end{table}

\end{document}